\documentclass[11pt]{amsart}
\usepackage{graphicx} \usepackage[margin=3cm]{geometry}
\usepackage{amsmath,amsfonts,amssymb,amsthm}
\usepackage{thmtools}
\usepackage{hyperref}
\usepackage{ulem}
\usepackage{xcolor}
\usepackage{shuffle}
\usepackage{cleveref}
\usepackage{subcaption}
\usepackage{tikz-cd}
\usepackage{microtype}
\usepackage{mathrsfs}
\usepackage{placeins}
\usepackage[T1]{fontenc}

\usepackage[notocbasic]{nomencl}

\newcounter{nomcount}
\newcommand{\nomkey}{\ifnum\value{nomcount}<10 00\fi
  \ifnum\value{nomcount}<100 \ifnum\value{nomcount}>9 0\fi\fi
  \arabic{nomcount}}
\newcommand{\autonom}[2]{\refstepcounter{nomcount}\nomenclature[\nomkey]{#1}{#2}}

\makenomenclature

\usepackage[utf8]{inputenc}
\usepackage[smalltableaux]{ytableau}
\hypersetup{
    colorlinks,
    linkcolor={red!50!black},
    citecolor={blue!50!black},
    urlcolor={black!50!black}
}

\newtheorem{theorem}{Theorem}[section]
\newtheorem*{theorem*}{Theorem}
\newtheorem{proposition}[theorem]{Proposition}
\newtheorem{lemma}[theorem]{Lemma}

\newtheorem{conjecture}[theorem]{Conjecture}
\newtheorem{corollary}[theorem]{Corollary}

\theoremstyle{definition}
\newtheorem{example}[theorem]{Example}
\newtheorem{remark}[theorem]{Remark}
\newtheorem{definition}[theorem]{Definition}

\newcommand{\paths}{\mathbf{PL}}

\newcommand{\vol}{\operatorname{vol}}
\newcommand{\proj}{\textnormal{proj}}
\newcommand{\spann}{\operatorname{span}}

\newcommand{\ann}{\operatorname{Ann}}

\newcommand{\emptyword}{\varnothing}
\newcommand{\Sym}{\operatorname{Sym}}
\newcommand{\AntiSym}{\operatorname{Alt}}
\newcommand{\comp}{\operatorname{Co}}
\newcommand{\tensor}{\mathbin{\otimes}}
\newcommand{\GL}{\operatorname{GL}}
\newcommand{\Pws}{\mathrm{PwS}}
\newcommand{\Pwa}{\mathrm{PwA}}
\newcommand{\KK}{\mathbb K}

\newcommand{\sigintegrand}[2]{\langle\omega(#1),#2\rangle}

\newcommand{\insertletter}{\operatorname{insert}}

\DeclareMathOperator{\im}{im}

\title{Piecewise Symmetric Tensors}

\author[Felix Lotter]{Felix Lotter$^1$}\thanks{$^1$Max Planck Institute for Mathematics in the Sciences, Leipzig, Germany}
\author[Rosa Preiß]{Rosa Preiß$^2$}\thanks{$^2$Technical University of Berlin, Germany}

\date{}

\begin{document}

\begin{abstract}
Every square matrix is uniquely the sum of a symmetric matrix and a skew-symmetric matrix.
We extend this familiar	fact to	higher-order tensors: every cubic $k$-tensor is uniquely the sum of an $m$-piecewise
symmetric tensor and a $(k\!-\!m)$-piecewise skew-symmetric tensor, for each choice of $m\leq k$. We study these tensor spaces via the combinatorics of word descents and compositions. Under Schur--Weyl duality, the two spaces naturally decompose into descent representations of the symmetric group. In fact, we prove a variant of Solomon's classical decomposition of its group algebra that explains this decomposition.

Our motivation
stems from signature tensors in stochastic analysis, algebraic geometry and data science: more specifically, we show that our tensor space decompositions determine the vanishing ideals of the signatures of piecewise linear paths with a fixed number of segments.

~\\
\textbf{Keywords}: symmetric tensors, alternating tensors, compositions, descents, ribbon Young diagrams, descent representations, path signatures, signed volumes, shuffles, vanishing ideals.

~\\
\textbf{MSC2020}: 05E10, 15A72, 20C30, 60L10
\end{abstract}

\maketitle

\section{Introduction}
Let $\KK \in \{\mathbb Q, \mathbb R, \mathbb C\}$\autonom{\(\KK\)}{The field of rational, real, or complex numbers}. Let us call a tensor $t \in (\KK^d)^{\tensor k}$ partially symmetric \textit{with $m$ connected blocks} if there is a composition $\alpha_1 + \ldots + \alpha_\ell = k$ of $k$ with $\alpha_1,\ldots,\alpha_\ell \in \mathbb N$ and $\ell \leq m$ such that $t \in \Sym^\alpha(\KK^d) := \Sym^{\alpha_1}(\KK^d) \tensor \dots \tensor \Sym^{\alpha_\ell}(\KK^d)$.

The study of partially symmetric tensors dates back to the 19th century (see \cite{bernardi2018hitchhiker} for an overview). Central questions concern their decompositions into simple tensors satisfying the same symmetries, which leads to the notion of partially symmetric rank \cite{ballico2019partially,wang2026multisubspacepowermethoddecomposing}.

In this article, we study the linear algebra and representation theory of tensors that arise by \textit{mixing} tensors of different partial symmetries, where we restrict to symmetries with connected blocks. Our main object of study here is the space of \textit{$m$-piecewise symmetric tensors}, which we define as follows: 
\[\Pws^{k}_{m}(\KK^d) := \sum_{
    \alpha \in \comp(k,\leq m)} \Sym^{\alpha}(\KK^d) \subseteq (\KK^d)^{\tensor k}.\]
    Here $\comp(k,\leq \!m)$ denotes the set of compositions of $k$ of length $\leq m$, i.e.\
    \begin{equation*}
        \comp(k,\leq \!m):=\{(\alpha_1,\dots,\alpha_\ell)\in\mathbb N^\ell\,|\,{\textstyle\sum_{i=1}^\ell\alpha_i=k}, \ \ell \leq m\}
    \end{equation*}
    where $\mathbb N=\{1,2,3,\dots\}$ denotes the positive integers. There are $\binom{k-1}{0} + \ldots + \binom{k-1}{m-1}$ such compositions. We write $\comp(k) := \comp(k,\leq \! k)$ for the set of all positive integer compositions of $k$.\autonom{\(\comp(k), \, \comp(k, \leq m)\)}{Positive integer compositions of $k$ (of length at most $m$)}

Our main result determines the annihilator of $\Pws^{k}_m(\KK^d)$. For a composition $\alpha = (\alpha_1,\ldots,\alpha_\ell)$ of $k$, let us write
$\AntiSym^\alpha(\KK^d) := \AntiSym^{\alpha_1}(\KK^d) \tensor \dots \tensor \AntiSym^{\alpha_\ell}(\KK^d)$ for the space of \textit{partially alternating} (or \textit{skew-symmetric}) tensors. Then we define the space of \textit{$m$-piecewise alternating} tensors as
\[\Pwa^{k}_{m}(\KK^d) := \sum_{
    \alpha \in \comp(k,\leq m)} \AntiSym^{\alpha}(\KK^d) \subseteq 
    (\KK^d)^{\tensor k}.\]

\begin{theorem}\label{thm:main}
        For any $d,k$ and $0 \leq m \leq k$, the subspaces $\Pws^{k}_m(\KK^d)$ and $\Pwa^{k}_{k-m}(\KK^d)$ of $(\KK^d)^{\tensor k}$ define an orthogonal decomposition of $(\KK^d)^{\tensor k}$ with respect to the standard inner product, in particular
    $$(\KK^d)^{\tensor k} = \Pws^{k}_m(\KK^d) \oplus \Pwa^{k}_{k-m}(\KK^d).$$
\end{theorem}

In \Cref{lmm:auto} we will give explicit bases for the two subspaces that are compatible with the inclusions $\Pws^{k}_m(\KK^d) \subseteq \Pws^{k}_{m'}(\KK^d)$ and $\Pwa^{k}_m(\KK^d) \subseteq \Pwa^{k}_{m'}(\KK^d)$ for $m \leq m'$.

\begin{example}
    If $m=1$, then $\Pwa^k_{k-1}(\KK^d)$ is spanned by tensors $e_{vijw} - e_{vjiw}$ for words $v$ and $w$ and letters $i$ and $j$, where as usual $e_{i_1\ldots i_k}$ denotes the tensor $e_{i_1}\tensor \dots \tensor e_{i_k}$. Since the simple transpositions generate the symmetric group, we immediately verify that $\Pwa^k_{k-1}(\KK^d)$ is indeed the orthogonal complement of $\Pws^k_1(\KK^d) = \Sym^k(\KK^d)$. In particular, our decomposition of tensors generalizes the decomposition of matrices into symmetric and antisymmetric part.
    
    Now let us consider the decomposition $(\KK^3)^{\tensor 4} = \Pws^4_2(\KK^3) \oplus \Pwa^4_2(\KK^3)$. $\Pws^4_2(\KK^3)$ has dimension $66$ and consequently $\Pwa^4_2(\KK^3)$ has dimension $15$. Let us decompose the tensor $a=16\,e_{1223}\in (\KK^3)^{\tensor 4}$ into $b \in \Pws^4_2(\KK^3)$ and $c \in \Pwa^4_2(\KK^3)$. We obtain
    \begin{align*}
        b &= 5 \sum_{\tau \in S_4} e_{\tau(1223)} - \sum_{\tau \in S_1 \times S_3} \left(6 e_{\tau(2123)} + 2 e_{\tau(3122)}\right)\\
        &\hphantom{=} -  \sum_{\tau \in S_2 \times S_2} \left(e_{\tau(1322)} + e_{\tau(2213)}\right) - \sum_{\tau \in S_3 \times S_1} \left(6 e_{\tau(1232)} + 2 e_{\tau(2231)}\right),\\
        c &= \sum_{\tau \in S_2 \times S_2} \mathrm{sgn}(\tau) \left(6e_{\tau(2132)} - 2e_{\tau(3221)}\right) - 2 \sum_{\tau \in S_1 \times S_3} \mathrm{sgn}(\tau) e_{\tau(2321)} \\
        &\hphantom{=}- 2 \sum_{\tau \in S_3 \times S_1} \mathrm{sgn}(\tau) e_{\tau(3212)}
    \end{align*}
    where $\tau(i_1,\ldots,i_4) := (i_{\tau(1)},\ldots,i_{\tau(4)})$.
    The symmetrizing sums that appear here already hint at the bases that we will construct in \Cref{lmm:auto}. 
\end{example}

\begin{remark}
    Partially symmetric tensors satisfy a stability with respect to the tensor product of vector spaces of the form
\begin{equation*}
    \Sym^{\alpha_1,\dots,\alpha_\ell}(\KK^d)\otimes\Sym^{\beta_1,\dots,\beta_{\ell'}}(\KK^d)=\Sym^{\alpha_1,\dots,\alpha_\ell,\beta_1,\dots,\beta_{\ell'}}(\KK^d)
\end{equation*}
while our construction satisfies the following convolution/binomial stability with respect to the tensor product of vector spaces:
\begin{equation*}
    \Pws_m^k(\KK^d)=\sum_{k'=1}^k\Pws_{m_1}^{k'}(\KK^d)\otimes\Pws_{m_2}^{k-k'}(\KK^d), 
\end{equation*}
for any $m_1+m_2=m$. In fact, this gives an alternative definition of partially symmetric tensors, if we add the initial condition $\Pws^k_1(\KK^d)=\Sym^k(\KK^d)$.
\end{remark}

The spaces $\Sym^{\alpha}(\KK^d)$ in the definition of $\Pws^k_m(\KK^d)$ can overlap non-trivially, as can the spaces $\AntiSym^{\alpha}(\KK^d)$ in the definition of $\Pwa^k_m(\KK^d)$. For example, $\Sym^\alpha(\KK^d) \subseteq \Sym^\beta(\KK^d)$ whenever the composition $\alpha$ is refined by the composition $\beta$, i.e., $\beta$ can be obtained by summing adjacent components of $\alpha$. In particular, the sum in the definition does not change if we restrict to compositions of length exactly $m$. In \Cref{sec:refine-dec} we will associate $\mathrm{GL}(\KK^d)$-invariant subspaces $V_\alpha(\KK^d), W_\alpha(\KK^d) \subseteq (\KK^d)^{\tensor k}$ to every composition $\alpha$ of $k$ such that
\[\Sym^\alpha(\KK^d) = \bigoplus_{\beta \leq \alpha} V_\beta(\KK^d) \quad \text{ and } \AntiSym^\alpha(\KK^d) = \bigoplus_{\beta\leq \alpha} W_\beta(\KK^d)\]
where $\beta \leq \alpha$ if the composition $\beta$ is refined by the composition $\alpha$. Then \Cref{thm:main} gives a finer decomposition of tensor space
\begin{equation}\label{eq:finer-dec}
    (\KK^d)^{\tensor k} = \bigoplus_{\alpha \in \comp(k,\leq m)} V_\alpha(\KK^d) \oplus  \bigoplus_{\alpha \in \comp(k,\leq k-m)} W_\alpha(\KK^d).
\end{equation}

We show that the spaces $W_\alpha(\KK^d)$ correspond under Schur--Weyl duality to the left ideals in Solomon's classical decomposition of the group algebra of the symmetric group \cite{SOLOMON1968220}. In \Cref{thm:grp-alg-decomp}, we prove a variant of Solomon's result in which the roles of symmetrizer and antisymmetrizer are partly reversed, which explains the decomposition \eqref{eq:finer-dec}. The representations given by the left ideals in this decomposition are known as \textit{descent representations} of the symmetric group \cite{adin2005descent, moustakas2023descent}. Their Frobenius characteristics are ribbon Schur functions \cite{gessel1984multipartite, billera2006decomposable}.

This viewpoint allows us to understand the spaces $V_\alpha(\KK^d)$ and $W_\alpha(\KK^d)$ as the images of Young symmetrizers associated to certain \textit{ribbon Young tableaux}. Using a result of Gessel \cite{gessel1984multipartite}, we show how they decompose into irreducibles (\Cref{cor:irred-decomp}). We compare this to the Thrall decomposition \cite{AMENDOLA2025} of tensor space.

The original motivation for our investigation of piecewise symmetric tensors (and the terminology) stems from the study of path signatures. The \textit{iterated-integral signature} of a path is an infinite collection of tensors, defined by iterated integration. It is a central object in the theory of rough paths in stochastic analysis.

Given a piecewise linear path $X$ with $m$ segments $v_1,\dots,v_m\in\mathbb R^d$, its signature can be shown to equal
\begin{equation*}
    \sigma(X)=\exp_{\otimes}(v_1)\otimes \cdots\otimes \exp_{\otimes}(v_m),
\end{equation*}
where exponentiation and products are taken in the completed tensor algebra $T((\mathbb R^d)):=\prod_{k=0}^\infty (\mathbb R^d)^{\otimes k}$.\autonom{\(\sigma(X)\)}{The signature of the path $X$} The projection of $\sigma(X)$ to $(\mathbb R^d)^{\tensor k}$ is the \textit{level $k$ signature tensor} $\sigma^{(k)}(X)$.\autonom{\(\sigma^{(k)}(X)\)}{The level $k$ signature of the path $X$} In \Cref{prop:span_of_signatures} we show that the linear hull of level $k$ signature tensors for paths of $m$ segments through $\mathbb R^d$ is exactly $\Pws_m^{k}(\mathbb R^d)$. \Cref{thm:main} then states that the linear relations among general signature tensors of such paths are precisely the elements of $\Pwa_{k-m}^{k}(\mathbb R^d)$. We give geometric intuition for  this result in \Cref{thm:conc_vol} and its discussion.

More generally, one can ask for all polynomial relations that hold among the entries of tensors
\begin{equation*}
    \sigma^{(k)}(X)\in (\mathbb R^d)^{\otimes k}
\end{equation*}
for any path $X$ with $m$ segments, where $d,k,m$ are fixed. The set of all such tensors is semi-algebraic. The polynomial relations define an algebraic variety in $\Pws^{k}_m(\mathbb R^d)$, the Zariski closure of this semi-algebraic set. While first results on these varieties were obtained in \cite{AFS18}, an exact description, let alone a description of the underlying semi-algebraic sets, seems out of reach so far.

However, knowledge of the linear relations also informs our understanding of the non-linear relations: due to the shuffle identity \cite{Ree58}, all degree $n$ homogeneous polynomial relations on signature $k$-tensors of $m$-segment piecewise linear paths appear as linear relations in the annihilator $\Pwa_{kn-m}^{kn}(\mathbb R^d)$ of $\Pws_m^{kn}(\mathbb R^d)$. See \Cref{sec:signatures} for more details. Forming the direct sum of all $\Pwa^k_{k-m}(\mathbb R^d)$ for fixed $d,m$, we obtain the vanishing ideal of the signatures $\sigma(X)$ of piecewise linear paths $X$ with $m$ segments. The precise structure of the ideal is studied in the final \Cref{sec:ideals}.

Our work lays the foundation for future investigation of piecewise symmetric tensors. One possible direction here would be the study of the \textit{$m$-piecewise symmetric rank} of a tensor $x \in \Pws^k_m(\KK^d)$; that is, the smallest number $r\in\mathbb N$ such that there are decomposable tensors $x_1,\ldots,x_r \in \Pws^k_m(\KK^d)$ with $x_1 + \ldots + x_r = x$. Note that for $m = k$, this is the tensor rank; for $m = 1$ it is the symmetric rank. However, it might also be reasonable to require that the decomposable tensors $x_1,\ldots,x_r$ are themselves partially symmetric with $m$ connected blocks. This leads to another notion of rank which is clearly bounded from below by the former.

The definition of $\Pws^k_m(\KK^d)$ also inspires the definition of a natural subvariety: indeed, the simple tensors in each of the spaces $\Sym^{\alpha_1}(\KK^d) \tensor \dots \tensor \Sym^{\alpha_m}(\KK^d)$ parametrize a subvariety (known as a Segre-Veronese variety \cite{catalisano2005higher}), and one may take the join of all these subvarieties for $\alpha \in \comp(k,\leq \!m)$. It would be interesting to understand this join. It contains the level $k$ signature variety of piecewise linear paths with $m$ segments that was examined in \cite{AFS18}.

\subsubsection*{Acknowledgments}
We thank Fulvio Gesmundo, Tim Seynnaeve and Daniele Taufer for helpful discussions, and Bernd Sturmfels for his advice and helpful suggestions regarding the presentation.
R.P.\ also thanks Leonard Schmitz and Clemens Hofstadler for valuable insights stemming from joint discussions. 
The authors acknowledge funding by the Deutsche Forschungsgemeinschaft (DFG, German Research Foundation)
– CRC/TRR 388 “Rough Analysis, Stochastic Dynamics and Related Fields” – Project A04, 516748464.

 The simplification of the expression in \Cref{cor:dim-of-pws} was suggested by GPT-5.5. After uploading the first version of the article to the arXiv, we prompted GPT-5.6 Luna, GPT-5.6 Sol and GLM-5.3 for a review of the article, and incorporated some of the hints. All writing is our own.  

\section{Piecewise symmetric and piecewise alternating tensors}

In this section, we prove \Cref{thm:main}. We do this in two steps: first, we show that $\Pws^{k}_m(\KK^d)$ and $\Pwa_{k-m}^{k}(\KK^d)$ are orthogonal. Then we will construct explicit bases for the two spaces, which will imply the statement by a simple dimension count.

For $\alpha \in \mathbb N^\ell$, let $\alpha_+$ denote the sum $\alpha_1 + \ldots + \alpha_\ell$, and let $|\alpha|$ denote the number of components $\ell$.\autonom{\(\alpha_+\)}{The total sum of an integer vector}\autonom{\(\lvert\alpha\rvert\)}{The length of an integer vector} For $k \in \mathbb N$ we define an order on $\comp(k)$ by letting $\alpha \leq \beta$ for $\alpha,\beta \in \comp(k)$ if $\alpha$ is refined by $\beta$. With this order, $\comp(k)$ is a boolean lattice. This can be seen directly: to a subset $\{k_1 < \ldots < k_\ell\}$ of $[k-1]$ we can associate the composition $k_1 + (k_2 - k_1) + \ldots + (k - k_\ell)$ of $k$ with length $\ell + 1$, and conversely a composition $\alpha$ of length $\ell$ defines a subset
\begin{equation}\label{def:J-bij}
    J(\alpha) := \{\alpha_1, \alpha_1+\alpha_2,\ldots, \alpha_1+\ldots + \alpha_{\ell-1}\}
\end{equation}
of $[k-1]$ of cardinality $\ell - 1$.\autonom{\(J(\alpha)\)}{The subset of $[k-1]$ corresponding to a composition $\alpha \in \comp(k)$} This defines an isomorphism of lattices between $\comp(k)$ and $2^{[k-1]}$.

\begin{remark}\label{rmk:sym_antisym_subset}
We always have inclusions $\Sym^{\alpha_+}(\KK^d) \subseteq \Sym^{\alpha}(\KK^d)$ and $\AntiSym^{\alpha_+}(\KK^d) \subseteq \AntiSym^{\alpha}(\KK^d)$. More generally, we have $\Sym^\alpha(\KK^d) \subseteq \Sym^\beta(\KK^d)$ and $\AntiSym^\alpha(\KK^d) \subseteq \AntiSym^\beta(\KK^d)$ whenever $\alpha \leq \beta$.
\end{remark}

\begin{lemma}\label{lmm:orth}
    For every $d,k$ and $m\leq k$, the spaces $\Pws_m^{k}(\KK^d)$ and $\Pwa_{k-m}^{k}(\KK^d)$ are orthogonal with respect to the standard inner product.
\end{lemma}
\begin{proof}
    Let $\alpha \in \comp(k,\leq\! m)$ and $\beta \in \comp(k,\leq\! k-m)$. Let $\gamma$ be the coarsest composition of $k$ refining both $\alpha$ and $\beta$, i.e. $\gamma := \alpha \lor \beta$. In particular, $J(\gamma) = J(\alpha) \cup J(\beta)$. Thus we have $|\gamma| \leq (m-1) + (k-m-1) + 1 = k - 1$, and so $\gamma$ has at least one block of size $\geq 2$. This implies that $\Sym^{\gamma}(\KK^d)$ and $\AntiSym^{\gamma}(\KK^d)$ are orthogonal. But due to Remark~\ref{rmk:sym_antisym_subset}, we have
    \[\Sym^{\alpha}(\KK^d) \subseteq \Sym^{\gamma}(\KK^d) \quad \text{and} \quad \AntiSym^{\beta}(\KK^d) \subseteq \AntiSym^{\gamma}(\KK^d),\]
    which concludes the proof.
\end{proof}

We will now define an automorphism
\[\rho_m: (\KK^d)^{\tensor k}\to (\KK^d)^{\tensor k}\] that maps a coordinate subspace to $\Pws^{k}_m(\KK^d)$ and its complement to $\Pwa^{k}_{k-m}(\KK^d)$. To this end, we introduce the following terminology:

Let $w:= (i_1,\ldots,i_k)$ be a length $k$ sequence of natural numbers $1,\ldots,d$, also called a \textit{word} over the alphabet $\{1,\ldots,d\}$. An index $1\leq j\leq k-1$ is a \textit{descent} of $w$ if $i_j > i_{j+1}$. Let $D(w) \subseteq [k-1]$ denote the set of descents of $w$.

For a length $k$ sequence $w$, let us denote the composition associated to $D(w)$ by $\delta(w)$\autonom{\(\delta(w)\)}{The descent composition of a word $w$} and the composition associated to $[k-1]\setminus D(w)$ by $\xi(w)$\autonom{\(\xi(w)\)}{The non-descent composition of a word $w$}. We call $\delta(w)$ the \textit{descent composition} and $\xi(w)$ the \textit{non-descent composition}. Note that $\delta(w)$ is the composition of $k$ given by the lengths of the maximal weakly increasing consecutive subsequences of $w$ and $\xi(w)$ is given by the lengths of maximal strictly decreasing consecutive subsequences. For example, the length $6$ sequence $(2,1,3,3,4,1)$ defines the descent composition $(1,4,1)$ and the non-descent composition $(2,1,1,2)$. Since $|D(w)| + |[k-1] \setminus D(w)| = k - 1$, we have $|\delta(w)| + |\xi(w)| = k + 1$.

 For a composition $\alpha = (\alpha_1,\ldots,\alpha_\ell)$ of $k$, set $S_\alpha := S_{\alpha_1} \times \dots \times S_{\alpha_\ell}$\autonom{\(S_\alpha\)}{The product of symmetric groups corresponding to $\alpha \in \comp(k)$} and let us define
 \begin{equation}\label{eq:symmetrizers}
     p_{\alpha} := \frac{1}{\alpha_1!\dots \alpha_\ell!}\sum_{\tau \in S_{\alpha}} \tau \quad \text{and} \quad q_{\alpha} := \frac{1}{\alpha_1!\dots \alpha_\ell!}\sum_{\tau \in S_{\alpha}} \mathrm{sgn}(\tau) \tau
 \end{equation}\autonom{\(p_\alpha\)}{The symmetrizer of $S_\alpha$, see \eqref{eq:symmetrizers}}\autonom{\(q_\alpha\)}{The antisymmetrizer of $S_\alpha$, see \eqref{eq:symmetrizers}}
in the group algebra $\KK[S_k]$. Then for $x \in (\KK^d)^{\tensor k}$, $x \mapsto x.p_\alpha$ and $x \mapsto x.q_\alpha$ are the projectors onto $\Sym^\alpha(\KK^d)$ and $\AntiSym^\alpha(\KK^d)$. Here, $S_k$ acts on $(\KK^d)^{\tensor k}$ from the right via $e_w \mapsto e_{\tau^{-1}(w)}$.

Now we define the map $\rho_m$ simply as follows\autonom{\(\rho_m\)}{The automorphism realizing the decomposition in \Cref{thm:main}}:
\begin{align*}
    \rho_m: (\KK^d)^{\tensor k}&\to (\KK^d)^{\tensor k} \\
    e_w &\mapsto \begin{cases}
        e_w.p_{\delta(w)} & \text{ if $|\delta(w)| \leq m$} \\
        e_{w}.q_{\xi(w)} & \text{ if $|\delta(w)| \geq m + 1$}
    \end{cases}
\end{align*}

Since $|\delta(w)| + |\xi(w)| = k + 1$, $e_w$ maps into $\Pws^{k}_m(\KK^d)$ if $|\delta(w)| \leq m$ and into $\Pwa^{k}_{k-m}(\KK^d)$ if $|\delta(w)| \geq m + 1$. Thus, in order to obtain \Cref{thm:main} it remains to prove the following lemma:

\begin{lemma}\label{lmm:auto}
    The map $\rho_m$ is an automorphism.
\end{lemma}
\begin{proof}
    For $|\delta(w)| \leq m$ note that any sequence $\tau(w)$ for $\tau \in S_{\delta(w)}$ is lexicographically greater than or equal to $w$. Thus, in the basis given by words ordered lexicographically, the matrix corresponding to the restricted map
    \begin{equation*}
        \spann\big\{e_w\big|\,|\delta(w)| \leq m\}\to (\KK^d)^{\otimes k}
    \end{equation*}
    is in column echelon form, and therefore injective.
    Similarly, for $|\delta(w)| \geq m+1$ every $\tau(w)$ for $\tau \in S_{\xi(w)}$ is lexicographically smaller than or equal to $w$. Thus, in the basis given by words ordered reverse lexicographically, the matrix corresponding to the restricted map
    \begin{equation*}
        \spann\big\{e_w\big|\,|\delta(w)| \geq m+1\}\to (\KK^d)^{\otimes k}
    \end{equation*}
    is in column echelon form, and therefore injective.
    As we already know that $\Pws^{k}_m(\KK^d)$ and $\Pwa^{k}_{k-m}(\KK^d)$ are orthogonal (\Cref{lmm:orth}), this shows injectivity of the endomorphism $\rho_m$ of a finite-dimensional vector space and we conclude.
\end{proof}

\begin{proof}[Proof (of \Cref{thm:main})]
    Combine \Cref{lmm:orth} and \Cref{lmm:auto}.
\end{proof}

\begin{remark}
    For a general field $\KK$, the proof can be adapted to show that for any finite-dimensional $\KK$-vector space $V = \KK^d$ we have an exact sequence
    \[\begin{tikzcd}
        0 \rar & \Pws^k_{m}(V) \rar["\iota"] & V^{\tensor k} \rar["x\mapsto \mathrm{ev}_x"] &(\Pwa^k_{k-m}(V^*))^* \rar & 0
    \end{tikzcd}\]
   where $\iota$ is the canonical inclusion and $\mathrm{ev}_x(f):=f(x)$
    for $x\in V^{\otimes k}$ and
    \begin{equation*}
        f\in \Pwa^k_{k-m}(V^*)\subseteq (V^*)^{\otimes k}=(V^{\otimes k})^*.
    \end{equation*}
    Indeed, as the dual of the inclusion $\Pwa^k_{k-m}(V^*)\to (V^*)^{\tensor k}$, the map $\mathrm{ev}$ is surjective, and \Cref{lmm:orth} shows that $\mathrm{ev} \circ \iota = 0$. Arguing as in \Cref{lmm:auto}, one can construct bases of $\Pws^k_m(V)$ and $\Pwa^k_{k-m}(V^*)$ indexed by words. More precisely, setting $S(w) := \{e_{\tau^{-1}(w)} \ | \ \tau \in S_{\delta(w)}\}$  and $T(w) :=\{\mathrm{sgn}(\tau) e^*_{\tau^{-1}(w)} \ | \ \tau \in S_{\xi(w)}\}$, we see that $b_w := \sum_{v\in S(w)} v$ for $|\delta(w)| \leq m$ form a basis of $\Pws^k_m(V)$ and $b'_w := \sum_{v \in T(w)} v$ for $|\delta(w)| \geq m + 1$ form a basis of $\Pwa^k_{k-m}(V^*)$. Thus a dimension count shows exactness of the sequence.
\end{remark}

\begin{corollary}\label{cor:dim-of-pws}
    The dimension of $\Pws^{k}_m(\KK^d)$ is
    \[\sum_{i=1}^{m} (-1)^{m-i} \binom{k}{m-i} \binom{id+k-1}{k},\]
    the number of length $k$-sequences of numbers $1,\ldots, d$ with at most $m-1$ descents. In particular, the dimension of $\Pwa^{k}_{k-m}(\KK^d)$ is the number of length $k$-sequences of numbers $1,\ldots, d$ with at least $m$ descents.
\end{corollary}
\begin{proof}
    The condition $|\delta(w)| \leq m$ in the definition of $\rho_m$ is equivalent to $w$ having at most $m-1$ descents. Write $\mathrm{ND}(k,m,d)$ for the number of length $k$ sequences of numbers $1,\ldots, d$ with at most $m-1$ descents.

    By \cite[Equation 3.3]{Carlitz1966} the number of such sequences with exactly $i-1$ descents (counted as $i$ descents in the convention of loc.\ cit.) is
    \[\sum_{j=0}^{i-1} (-1)^j \binom{k+1}{j} \binom{(i-j)d+k-1}{k}.\]
    Thus we obtain
    \[\mathrm{ND}(k,m,d) = \sum_{i=1}^m \sum_{j=0}^{i-1} (-1)^j \binom{k+1}{j} \binom{(i-j)d+k-1}{k}.\]
    Switch the sums and substitute $i := i - j$ to get
    \begin{align*}
        \mathrm{ND}(k,m,d) &= \sum_{j=0}^{m-1} \sum_{i=j+1}^m  (-1)^j \binom{k+1}{j} \binom{(i-j)d+k-1}{k} \\
        &= \sum_{j=0}^{m-1} \sum_{i=1}^{m-j} (-1)^j \binom{k+1}{j} \binom{i d+k-1}{k}.
    \end{align*}
    Now switch the sums back. This yields
    \[\sum_{i=1}^{m} \left( \sum_{j=0}^{m-i} (-1)^{j} \binom{k+1}{j}\right) \binom{i d+k-1}{k}  = \sum_{i=1}^{m} (-1)^{m-i} \binom{k}{m-i} \binom{i d+k-1}{k}.\]
\end{proof}
\begin{corollary}\label{cor:space-spanning-bound}
    We have $\Pws_m^k(\KK^d) = (\KK^d)^{\tensor k}$ if and only if $m\geq \lfloor \frac{d-1}{d} k\rfloor+1$.
\end{corollary}
\begin{proof}
    By \Cref{cor:dim-of-pws}, $\Pws_m^k(\KK^d) = (\KK^d)^{\tensor k}$ if and only if every length $k$ sequence of numbers $1,\ldots,d$ has at most $m - 1$ descents. The number of descents is maximized by the length $k$ sequence $(d,d-1,\ldots,1,d,d-1,\ldots)$ with $\lfloor (k-1) - (\frac{k}{d}-1) \rfloor =\lfloor \frac{d-1}{d} k\rfloor$ descents. We obtain the sharp bound $\lfloor \frac{d-1}{d} k \rfloor \leq m - 1$.
\end{proof}

\begin{example}
    \Cref{tab:dim-exmp} shows the dimensions of $\Pws^k_m(\KK^2)$ and $\Pws^k_m(\KK^3)$ for different choices of $k$ and $m$.

    The column sequences for $m = 2$ and $m=3$ can be found in the OEIS \cite{OEIS} as the first rows of \texttt{A255992} (\texttt{A000125}, `cake numbers') and \texttt{A256816} for $d=2$, and the first rows of \texttt{A255107} and \texttt{A255622} for $d=3$.
    
    \begin{table}[h]
        \centering
        \begin{subtable}{0.45\linewidth}
        \begin{tabular}{c|ccccc}
            $k \backslash m$ & 1 & 2 & 3 & 4 & 5 \\ \hline
            2 & 3 & \underline 4 &  &  &  \\
            3 & 4 & \underline 8 & &  &  \\
            4 & 5 & 15 & \underline{16} &  &  \\
            5 & 6 & 26 & \underline{32} &  & \\
            6 & 7 & 42 & 63 & \underline{64} & \\
            7 & 8 & 64 & 120 & \underline{128} & \\
            8 & 9 & 93 & 219 & 255 & \underline{256} \\
        \end{tabular}
        \subcaption{Dimensions of $\Pws^k_m(\KK^2)$}
        \end{subtable}
        \begin{subtable}{0.45\linewidth}
        \begin{tabular}{c|cccccc}
            $k \backslash m$ & 1 & 2 & 3 & 4 & 5 & 6 \\ \hline
            2 & 6 & \underline 9 &  &  & & \\
            3 & 10 & 26 & \underline{27} &  & & \\
            4 & 15 & 66 & \underline{81} &  & & \\
            5 & 21 & 147 & 237 & \underline{243}  & & \\
            6 & 28 & 294 & 651 & 728 & \underline{729} & \\
            7 & 36 & 540 & 1647 & 2151 & \underline{2187} & \\
            8 & 45 & 927 & 3834 & 6138 & 6552 & \underline{6561} \\
        \end{tabular}
        \subcaption{Dimensions of $\Pws^k_m(\KK^3)$}
        \end{subtable}
        \caption{Dimensions of spaces of piecewise symmetric tensors. We omit the entries in the rows once the ambient dimensions ($d^k$, underlined) are attained (cf.\ \Cref{cor:space-spanning-bound}).}
        \label{tab:dim-exmp}
    \end{table}

\end{example}

\begin{lemma}\label{lmm:sym-bases}
    Let $\alpha \in \comp(k)$. The space $\Sym^{\alpha}(\KK^d)$ is spanned by the basis tensors $\rho_k(e_w)$ for sequences $w$ such that $\delta(w) \leq \alpha$. Similarly, $\AntiSym^{\alpha}(\KK^d)$ is spanned by the basis tensors $\rho_0(e_w)$ for $w$ such that $\xi(w) \leq \alpha$.
\end{lemma}
\begin{proof}
    By definition of $\rho_k$, $\rho_k(e_w) \in \Sym^{\alpha}(\KK^d)$ for every $w$ such that $\alpha$ refines $\delta(w)$ and $\rho_0(e_w) \in \AntiSym^{\alpha}(\KK^d)$ for every $w$ such that $\alpha$ refines $\xi(w)$. Thus, by \Cref{thm:main} it suffices to count dimensions. For a sequence $w$, $\alpha$ is a refinement of $\delta(w)$ if and only if on the parts of the partition $\alpha$, $w$ is weakly increasing. But the number of weakly increasing length $\alpha_i$ sequences of numbers $1,\ldots ,d$ is precisely the dimension of $\Sym^{\alpha_i}(\KK^d)$, which implies the statement.\par
    Similarly, $\alpha$ is a refinement of $\xi(w)$ if and only if $w$ is strictly decreasing on the parts of the partition $\alpha$, and again, the number of strictly decreasing length $\alpha_i$ sequences of numbers $1,\ldots, d$ is precisely the dimension of $\AntiSym^{\alpha_i}(\KK^d)$. Thus, we conclude.
\end{proof}

\begin{corollary}\label{cor:sym-intersections}
    For compositions $\alpha^{(1)},\ldots,\alpha^{(n)}$, we have \[\Sym^{\alpha^{(1)}}(\KK^d) \cap \ldots \cap \Sym^{\alpha^{(n)}}(\KK^d) = \Sym^{\alpha^{(1)} \land \ldots \land \alpha^{(n)}}(\KK^d)\]
    and \[\AntiSym^{\alpha^{(1)}}(\KK^d) \cap \ldots \cap \AntiSym^{\alpha^{(n)}}(\KK^d) = \AntiSym^{\alpha^{(1)} \land \ldots \land \alpha^{(n)}}(\KK^d)\]
    where $\alpha^{(1)} \land \ldots \land \alpha^{(n)}$ denotes the meet in the lattice $\comp(k)$, that is, the finest composition refined by all $\alpha^{(i)}$.

    In particular, if $d\geq k$ then the intersection lattice of the subspace arrangements given by the $\Sym^\alpha(\KK^d)$ resp.\ $\AntiSym^\alpha(\KK^d)$ is isomorphic to the lattice $\comp(k)$.
\end{corollary}
\begin{proof}
    By \Cref{lmm:sym-bases} the intersection \[\Sym^{\alpha^{(1)}}(\KK^d) \cap \ldots \cap \Sym^{\alpha^{(n)}}(\KK^d)\] is spanned by the tensors $\rho_k(e_w)$ with $\delta(w) \leq \alpha^{(1)},\ldots,\alpha^{(n)}$. Those are precisely the tensors $\rho_k(e_w)$ with $\delta(w) \leq \alpha^{(1)} \land \ldots \land \alpha^{(n)}$ which yields the statement by another application of \Cref{lmm:sym-bases}. The same argument applies with $\rho_k$ replaced by $\rho_0$ and $\delta(w)$ replaced by $\xi(w)$ to show the statement about the spaces $\AntiSym^{\alpha}(\KK^d)$.
\end{proof}

\section{Descent representations}\label{sec:refine-dec}

In this section, we set out to answer the following three questions about the structure of piecewise symmetric and piecewise alternating tensors:
\begin{enumerate}
    \item The spaces $\Sym^\alpha(\KK^d)$ and $\AntiSym^\alpha(\KK^d)$ in the sums defining $\Pws^k_m(\KK^d)$ and $\Pwa^k_m(\KK^d)$ overlap. What are the precise contributions of $\Sym^\alpha(\KK^d)$ and $\AntiSym^\alpha(\KK^d)$ to these sums for varying $\alpha$?
    \item The spaces $\Pws^k_m(\KK^d)$ and $\Pwa^k_m(\KK^d)$ are $\mathrm{GL}(\KK^d)$-representations. How do they decompose into irreducible representations?
    \item How can we compute the projectors to $\Pws^k_m(\KK^d)$ and $\Pwa^k_m(\KK^d)$?
\end{enumerate}

We start with the first question, which will lead to a refinement of our decomposition. By \Cref{cor:sym-intersections}, the intersections of the subspaces $\Sym^\alpha(\KK^d)$ (resp.\ $\AntiSym^\alpha(\KK^d)$) for different compositions $\alpha$ are governed by the lattice $\comp(k)$. Working upwards through this lattice, we can write every $\Sym^\alpha(\KK^d)$ and $\AntiSym^\alpha(\KK^d)$ and consequently $\Pws^k_m(\KK^d), \Pwa^k_{k-m}(\KK^d)$ as a direct sum of spaces $V_\beta$ and $W_\beta$ that we associate to compositions $\beta \leq \alpha$.

Indeed, for a composition $\alpha$ of $k$, recall the idempotents $p_\alpha$ and $q_\alpha$ in the group algebra $\KK[S_k]$ from \eqref{eq:symmetrizers}. Let $R: \KK[S_k] \to \mathrm{End}^{\mathrm{op}}((\KK^d)^{\tensor k})$ be the algebra homomorphism corresponding to the right action of $S_k$ on $(\KK^d)^{\tensor k}$. Then $R(p_\alpha)$ and $R(q_\alpha)$ are the projections onto $\Sym^\alpha(\KK^d)$ and $\AntiSym^\alpha(\KK^d)$, respectively.

Now  we set $V_{(k)}(\KK^d) := \Sym^k(\KK^d)$, $W_{(k)}(\KK^d) := \AntiSym^k(\KK^d)$ and for all compositions $\alpha \in \comp(k)$ with $\alpha > (k)$: 
\[V_\alpha(\KK^d) := \bigcap_{\beta < \alpha}\ker R(p_\beta)|_{\Sym^\alpha(\KK^d)} \quad \text{and} \quad W_\alpha(\KK^d) := \bigcap_{\beta < \alpha}\ker R( q_\beta)|_{\AntiSym^\alpha(\KK^d)}.\]
That is, $V_\alpha(\KK^d)$ (resp.\ $W_\alpha(\KK^d)$) is the space of all $x \in \Sym^\alpha(\KK^d)$ (resp.\ $x \in \AntiSym^\alpha(\KK^d)$) that vanish under the projection to any $\Sym^\beta(\KK^d)$ (resp.\ $\AntiSym^\beta(\KK^d)$) for $\beta < \alpha$.
Note that since $R(p_{\alpha})$ and $R(q_{\alpha})$ are $\GL(\KK^d)$-equivariant, $V_\alpha(\KK^d)$ and $W_\alpha(\KK^d)$ are again $\GL(\KK^d)$-invariant.
\begin{proposition}\label{prop:decomp-partial-sym}
    For every composition $\alpha$ of $k$, we have
    \[\Sym^{\alpha}(\KK^d) = \bigoplus_{\beta \leq \alpha} V_\beta(\KK^d), \qquad \AntiSym^{\alpha}(\KK^d) = \bigoplus_{\beta \leq \alpha} W_\beta(\KK^d)\]
    as subspaces of $(\KK^d)^{\tensor k}$.
\end{proposition}
\begin{proof}
    We prove the statement by induction on the length $\ell$ of $\alpha$. If $\ell = 1$, then $\alpha = (k)$ and the statement holds by definition. Now assume $\ell > 1$ and that the statement holds for the compositions of length $\leq \ell -1$. Consider the sum $Q := \sum_{\beta < \alpha} \Sym^\beta(\KK^d) = \sum_{\beta < \alpha} p_\beta(\Sym^{\alpha}(\KK^d)) \subseteq \Sym^\alpha(\KK^d)$. By induction, $Q = \bigoplus_{\beta < \alpha} V_\beta(\KK^d)$. Now note that
    \[Q^\perp = \Bigl(\sum_{\beta < \alpha} p_\beta(\Sym^{\alpha}(\KK^d))\Bigr)^\perp = \bigcap_{\beta < \alpha} p_\beta(\Sym^{\alpha}(\KK^d))^\perp\]
    where the orthogonal complement is taken in $\Sym^\alpha(\KK^d)$ with respect to the standard inner product on $(\KK^d)^{\tensor k}$. But since the standard inner product is invariant under the action of $S_k$, the symmetrizers $p_\beta$ are self-adjoint and thus $p_\beta(\Sym^{\alpha}(\KK^d))^\perp = \ker p_\beta|_{\Sym^\alpha(\KK^d)}$. So $Q^\perp = V_\alpha(\KK^d)$ which completes the induction step. The same argument applies to show the statement about the $W_\beta(\KK^d)$.
\end{proof}
\begin{corollary}\label{cor:tot-decomp}
    We have
    \[(\KK^d)^{\tensor k} = \bigoplus_{\alpha \in \comp(k)} V_{\alpha}(\KK^d) = \bigoplus_{\alpha \in \comp(k)} W_{\alpha}(\KK^d)\]
\end{corollary}
\begin{proof}
    Apply \Cref{prop:decomp-partial-sym} to $\alpha = (1,\ldots, 1)$.
\end{proof}

\begin{corollary}\label{cor:refine-decomp}
    The decomposition $(\KK^d)^{\tensor k} = \Pws^k_m(\KK^d) \oplus \Pwa^k_{k-m}(\KK^d)$ from \Cref{thm:main} is refined by
    \[(\KK^d)^{\tensor k} = \bigoplus_{\alpha \in \comp(k,\leq m)} V_\alpha(\KK^d) \ \oplus \ \bigoplus_{\alpha \in \comp(k,\leq k - m)} W_\alpha(\KK^d) \]
\end{corollary}
\begin{proof}
    This follows from \Cref{thm:main} and \Cref{prop:decomp-partial-sym} as the sets $\comp(k,\leq \! i)$ are stable under coarsenings, i.e., lower sets in the lattice $\comp(k)$.
\end{proof}
\autonom{\(V_\alpha(\KK^d), \, W_\alpha(\KK^d)\)}{The spaces appearing in the refined decomposition in \Cref{cor:refine-decomp}.}

Since all the spaces $V_\alpha(\KK^d)$ and $W_\alpha(\KK^d)$ are $\mathrm{GL}(\KK^d)$-invariant, by Schur-Weyl duality there is a decomposition of the group algebra $\KK[S_k]$ corresponding to the one in \Cref{cor:refine-decomp}. We will now show that this decomposition is a variant of Solomon's classical decomposition of the group algebra into left ideals \cite{SOLOMON1968220, SOLOMON1976255}. 
If $\alpha=(\alpha_1,\dots,\alpha_\ell)$ is a composition of $k$, recall from \eqref{def:J-bij} that we write
\[J(\alpha) = \{\alpha_1,\alpha_1+\alpha_2,\ldots,\alpha_1+\ldots+\alpha_{\ell-1}\}\]
for the corresponding subset of $[k-1]$. Let $\hat J(\alpha) := [k-1] \setminus J(\alpha)$ denote the complement,
and $\hat \alpha$ the corresponding composition so that $J(\hat \alpha) = \hat J(\alpha)$.
Notice that $S_\alpha = S_{\alpha_1}\times\dots\times S_{\alpha_\ell}$ is then equivalently the subgroup of $S_k$ generated by $s_i$ for $i\in \hat J(\alpha)$.
 In \cite[Theorem 2]{SOLOMON1968220}, Solomon proves the inner direct sum decomposition
\[\KK[S_k] = \bigoplus_{\alpha} \KK[S_k]p_{\hat \alpha} q_{\alpha} \]
of the group algebra into left ideals, where $\alpha$ runs over all compositions of $k$. One can show a slightly more flexible version of this statement, making use of the arguments in \cite{SOLOMON1968220}:

\begin{theorem}\label{thm:grp-alg-decomp}
    For $1\leq m\leq k$ we have an inner direct sum decomposition
    \[\KK[S_k] = \bigoplus_{\alpha \in \comp(k,\leq m)} \KK[S_k]  q_{\hat \alpha} p_\alpha \ \oplus \ \bigoplus_{\alpha \in \comp(k,\leq k - m)} \KK[S_k]  p_{\hat \alpha} q_\alpha \]
    of the group algebra $\KK[S_k]$ of the symmetric group into left ideals.
\end{theorem}
\begin{proof}
    In the following we assume $\mathbb K = \mathbb Q$ without loss of generality. We recall the symmetric, positive definite bilinear form $F$ from the proof of \cite[Lemma 2]{SOLOMON1968220}. For $b \in \KK[S_k]$, write $T(b)$ for the $\KK$-linear endomorphism of $\KK[S_k]$ given by right multiplication with $b$. Let $r: \KK[S_k] \to \KK[S_k]$ be the involutory antiautomorphism induced by mapping $\tau \in S_k$ to $\tau^{-1}$. Then one sets $F(a,b):=\mathrm{trace}\ T(r(a)b)$.

    We claim that $\KK[S_k]  q_{\hat \alpha} p_\alpha$ and $\KK[S_k]  p_{\hat \beta} q_\beta$ are orthogonal with respect to $F$ whenever $\alpha \in \comp(k,\leq m)$ and $\beta \in \comp(k,\leq k-m)$. For this, note that $T(p_{\alpha})$ is self-adjoint with respect to $F$ since $p_\alpha = r(p_\alpha)$. Thus, 
    \[F(\KK[S_k]  q_{\hat \alpha} p_\alpha, \KK[S_k]  p_{\hat \beta} q_\beta) = F(\KK[S_k]  q_{\hat \alpha}, \KK[S_k]  p_{\hat \beta} q_\beta p_\alpha) \]
    but since $|\alpha| + |\beta| \leq k$ we have $| J(\alpha)| +  |J(\beta)| \leq k - 2$ and thus
    \[|J(\hat \alpha)| + |J(\hat \beta)| \geq k>k-1\] which implies $J(\hat \alpha) \cap J(\hat \beta) \not= \emptyset$; whence $q_\beta p_\alpha = 0$ by \cite[Lemma 1]{SOLOMON1968220}.

    Finally, directness of the sums
    \[\sum_{\alpha \in \comp(k,\leq m)} \KK[S_k]  q_{\hat \alpha} p_\alpha \quad \text{and} \quad \sum_{\alpha \in \comp(k,\leq k - m)} \KK[S_k]  p_{\hat \alpha} q_\alpha\]
    is shown in \cite[Lemma 1, 2 and 3]{SOLOMON1968220}. To be precise, Solomon proves directness of the latter sum for all $m$, which implies directness of the former sum by applying the sign-twist automorphism $\KK[S_k] \to \KK[S_k], \ \tau \mapsto \mathrm{sgn}(\tau)\tau$.

    We can now conclude by a dimension count: there are isomorphisms
    \[\KK[S_k]q_{\hat \alpha} p_{\alpha} \cong \KK[S_k]p_{\alpha} q_{\hat \alpha}\]  by \cite[Lemma 12]{SOLOMON1968220} and since $\alpha \in \comp(k,\leq m)$ if and only if $\hat \alpha \in \comp(k,>k-m)$, they identify the direct sum in the statement with the one that is proven to be a decomposition in \cite[Theorem 2]{SOLOMON1968220}.
\end{proof}

\begin{example}\label{ex:id-decomp-3}
    Let us decompose $1 = \sum_{\alpha \in \comp(3,\leq 3)} e_\alpha \in \KK[S_3]$ with $e_\alpha \in \KK[S_3] q_{\hat \alpha} p_\alpha$, according to \Cref{thm:grp-alg-decomp}. We have
    \begin{align*}
        q_{(1,1,1)} p_{(3)} &= \frac{1}{6}(1 + (12) + (23) + (13) + (123) + (321)) \\
        q_{(1,2)} p_{(2,1)} &= \frac{1}{4} (1 - (23)) (1 + (12)) = \frac{1}{4} (1 + (12) - (23) - (132)) \\
        q_{(2,1)} p_{(1,2)} &= \frac{1}{4} (1 - (12)) (1 + (23)) = \frac{1}{4} (1 - (12) + (23) - (123)) \\
        q_{(3)} p_{(1,1,1)} &= \frac{1}{6}(1 - (12) - (23) - (13) + (123) + (321))
    \end{align*}
    and thus we obtain $e_{(3)} = q_{(1,1,1)} p_{(3)}$, $e_{(1,1,1)} = q_{(3)} p_{(1,1,1)}$, $e_{(2,1)} = \frac{4}{3} q_{(1,2)}p_{(2,1)}$ and $e_{(1,2)} = \frac{4}{3} q_{(2,1)}p_{(1,2)}$. Note that since the $e_\alpha$ decompose $1$, they are necessarily orthogonal idempotents: $e_\alpha = e_\alpha \cdot 1 = \sum_\beta e_\alpha e_\beta$ and as the sum $\sum_\alpha \KK[S_3] q_{\hat \alpha} p_\alpha$ is direct this shows $e_\alpha e_\beta = \delta_{\alpha\beta} e_\alpha$.
\end{example}

We also have an analogue of \Cref{prop:decomp-partial-sym}:

\begin{proposition}\label{prop:left-ideal-decomp}
    Let $\alpha$ be a composition of $k$. Then
    \[\KK[S_k]p_\alpha = \bigoplus_{\beta \leq \alpha} \KK[S_k]q_{\hat \beta}p_\beta \quad \text{and} \quad \KK[S_k]q_\alpha = \bigoplus_{\beta \leq \alpha} \KK[S_k]p_{\hat \beta}q_\beta.\]
\end{proposition}
\begin{proof}
    We prove the first statement, the second is analogous.

    First, note that for $\beta \leq \alpha$ we have $\KK[S_k]q_{\hat \beta}p_\beta p_\alpha = \KK[S_k]q_{\hat \beta} p_\beta$ and thus $\KK[S_k]q_{\hat \beta} p_\beta \subseteq \KK[S_k]p_\alpha$. By \Cref{thm:grp-alg-decomp}, the inner sum is direct. By \cite[Lemma 4]{SOLOMON1968220}, we have a decomposition
    \[\KK[S_k]p_\alpha = \bigoplus_{\beta \leq \alpha} \KK[S_k]p_{\beta}q_{\hat \beta} p_\alpha\]
    and in the proof of \cite[Lemma 5]{SOLOMON1968220} it is shown that $\KK[S_k]p_{\beta}q_{\hat \beta} p_\alpha \cong \KK[S_k]p_{\beta}q_{\hat \beta}$. Finally, by \cite[Lemma 12]{SOLOMON1968220} $\KK[S_k]p_{\beta}q_{\hat \beta} \cong \KK[S_k]q_{\hat \beta}p_{\beta}$. Thus, the claim follows by a dimension count.
\end{proof}

We can view $(\KK^d)^{\tensor k}$ as a right $\KK[S_k]$ module via $R: \KK[S_k] \to \mathrm{End}^{\mathrm{op}}((\KK^d)^{\tensor k})$ corresponding to the right action. Then from \Cref{thm:grp-alg-decomp} we obtain an isomorphism of $\GL(\KK^d)$-representations
\begin{equation}\label{eq:tensor-id}
    \begin{split}
    (\mathbb K^d)^{\tensor k} \simeq (\mathbb K^d)^{\tensor k} \tensor_{\KK[S_k]} \KK[S_k] \simeq \bigoplus_{\alpha \in \comp(k,\leq m)} (\mathbb K^d)^{\tensor k} \tensor_{\KK[S_k]}  \KK[S_k]q_{\hat \alpha} p_{\alpha} \\
    \oplus \bigoplus_{\alpha \in \comp(k,\leq k-m)} (\mathbb K^d)^{\tensor k} \tensor_{\KK[S_k]}  \KK[S_k]p_{\hat \alpha} q_{\alpha} 
    \end{split}
\end{equation}

It remains to show that this is precisely the decomposition in \Cref{cor:refine-decomp}:

\begin{proposition}\label{prop:descent-alt}
    Under the identification \eqref{eq:tensor-id}, we have
    \[V_\alpha(\KK^d) \simeq (\mathbb K^d)^{\tensor k} \tensor_{\KK[S_k]}  \KK[S_k]q_{\hat \alpha} p_{\alpha}\]
    and
    \[W_\alpha(\KK^d) \simeq (\mathbb K^d)^{\tensor k} \tensor_{\KK[S_k]}  \KK[S_k]p_{\hat \alpha} q_{\alpha}.\]
\end{proposition}
\begin{proof}
    The isomorphism \eqref{eq:tensor-id} identifies $(\mathbb K^d)^{\tensor k} \tensor_{\KK[S_k]}  \KK[S_k]p_{\hat \alpha} q_{\alpha}$ with the subspace $U$ of $(\mathbb K^d)^{\tensor k}$ spanned by tensors $e_w.p_{\hat \alpha}.q_{ \alpha}$. By definition, this is a subspace of $\AntiSym^\alpha(\KK^d)$. Moreover, for $\beta < \alpha$, we have $R(q_\beta)(e_w.p_{\hat \alpha}.q_{ \alpha}) = e_w.(p_{\hat \alpha} q_{ \alpha} q_\beta)$. But since $\beta < \alpha$, $q_\alpha q_\beta = q_\beta q_\alpha$, and since in particular $J( \alpha) \cap \hat J(\beta) \not=\emptyset$, $p_{\hat \alpha} q_\beta = 0$ by \cite[Lemma 1]{SOLOMON1968220}. Thus $W_\alpha(\KK^d) \supseteq U \simeq (\mathbb K^d)^{\tensor k} \tensor_{\KK[S_k]}  \KK[S_k]p_{\hat \alpha} q_{\alpha}$. The analogous argument with $p$ and $q$ reversed shows that we have an inclusion $V_\alpha(\KK^d) \supseteq (\mathbb K^d)^{\tensor k} \tensor_{\KK[S_k]}  \KK[S_k]q_{\hat \alpha} p_{\alpha}$. Now the claim follows by dimension count since by \Cref{cor:refine-decomp} and \Cref{thm:grp-alg-decomp} both vector spaces decompose $(\KK^d)^{\otimes k}$ when summed over all $\alpha$.
\end{proof}

In particular, $V_\alpha(\KK^d) = \im( R(p_{\alpha}) \circ R(q_{\hat \alpha}))$ and $W_\alpha(\KK^d) = \im (R(q_{\alpha}) \circ R(p_{\hat \alpha}))$. These are analogues of Young symmetrizers: indeed, compositions of $k$ are in bijection with skew Young diagrams with $k$ boxes which do not contain a $2\times 2$ square. Such diagrams are also called \textit{zig-zag diagrams} or \textit{ribbons}. The skew Young diagram $Y(\alpha)$ associated to a composition $\alpha$ of length $\ell$ has $\alpha_i$ boxes in row $\ell + 1 - i$, such that the last box in row $i$ is below the first box in row $i-1$. For example, the composition $(2,3,1)$ corresponds to the diagram
\[\ydiagram{3+1,1+3,2}\]
If we flip $Y(\alpha)$ vertically, then $R(p_{\alpha}) \circ R(q_{\hat \alpha})$ is the analogue of a Young symmetrizer for the unique standard tableau for this flipped diagram, as Solomon already explained in the original article \cite{SOLOMON1968220}. 

\begin{example}
    Consider $k=4$. The partition $(2,2)$ corresponds to the following diagram and unique tableau:
    \[\ydiagram{2,1+2} \qquad \ytableaushort{12,\none 34}\]
    Thus we have to antisymmetrize on positions $2$ and $3$, followed by symmetrizing positions $1,2$ and $3,4$. This will yield a surjection onto $V_{(2,2)}(\KK^d)$.\par
    Note however that these surjections are not necessarily the orthogonal idempotents realizing the decomposition.
\end{example}

Combining \Cref{prop:descent-alt} with \Cref{thm:grp-alg-decomp}, we see that we obtain the projectors to the spaces $V_\alpha(\KK^d)$ and $W_\alpha(\KK^d)$ by decomposing $1 \in \KK[S_k]$ according to the decomposition from \Cref{thm:grp-alg-decomp}. This is a linear algebra problem. By \Cref{cor:refine-decomp} we then obtain the projectors to $\Pws^k_m(\KK^d)$ and $\Pwa^k_{k-m}(\KK^d)$ by taking the appropriate sums.

\begin{example}
    From \Cref{ex:id-decomp-3} we obtain the obvious projectors $R(q_{(1,1,1)}p_{(3)}) = R(p_{(3)}) $ to $\Pws^3_1(\KK^d)=\Sym^3(\KK^d)$ and \[R\Big(q_{(1,1,1)}p_{(3)} + \frac{4}{3} (q_{(1,2)}p_{(2,1)} + q_{(2,1)}p_{(1,2)})\Big) = \mathrm{id} - R(q_{(3)})\] to $\Pws^3_2(\KK^d)$. The first interesting examples are the projectors $P$ to $\Pws^4_2(\KK^d)$ and $Q$ to $\Pwa^4_2(\KK^d)$. Note that $Q = \mathrm{id} - P$. We compute
    \begin{scriptsize}
    \begin{align*}
        P &= R\Big(\frac{1}{2} \cdot 1 + \frac{1}{4}(34) + \frac{1}{12}(23) + \frac{1}{24}(234) - \frac{1}{24}(243) + \frac{1}{6}(24) + \frac{1}{4}(12) - \frac{1}{24}(123) - \frac{1}{12}(1234) - \frac{1}{6}(1243) \\
        &\hphantom{:=} + \frac{1}{24}(124) + \frac{1}{24}(132) - \frac{1}{6}(1342) + \frac{1}{6}(13) - \frac{1}{24}(134) - \frac{1}{12}(1432)
        - \frac{1}{24}(142) + \frac{1}{24}(143) + \frac{1}{12}(14)\Big).
    \end{align*}
    \end{scriptsize}
\end{example}

Finally, we want to understand how the representation $\KK[S_k] q_{\hat \alpha} p_{\alpha}\cong \KK[S_k]p_{\alpha}q_{\hat\alpha}$ decomposes into irreducible representations. For this, we use the following fact:
\begin{lemma}\label{lem:desc-rep}
    $\KK[S_k] q_{\hat \alpha} p_{\alpha}\cong \KK[S_k]p_{\alpha}q_{\hat\alpha}$ is the $S_k$-representation associated to the skew Young diagram $Y(\alpha)$.
\end{lemma}
\begin{proof}
    It follows from \Cref{prop:left-ideal-decomp} and Moebius inversion that the character of $\KK[S_k] q_{\hat \alpha} p_{\alpha}$ (which is also the character of $\KK[S_k]p_{ \alpha}q_{\hat\alpha} $ by \cite[Lemma 12]{SOLOMON1968220}) is given by
\[\psi_{\alpha} = \sum_{\beta \leq \alpha} (-1)^{|\beta|-|\alpha|}\phi_\beta\]
where $\phi_\beta$ is the character of $\KK[S_k]p_\beta$. In \cite[Section 2]{gessel1984multipartite} it is shown that $\psi_\alpha$ is identified under the Frobenius characteristic map with the skew Schur function associated to $Y(\alpha)$, given by the sum
\[s_{Y(\alpha)} := \sum_{T \in \mathrm{SSYT}(Y(\alpha))} x^T \in \KK[x_1,\dots,x_k]\]
over semi-standard Young tableaux of shape $Y(\alpha)$ with entries in $\{1,\ldots,k\}$.
\end{proof}

The representations associated to ribbon Young diagrams are known as \textit{ribbon} or \textit{descent representations}. We refer to \cite[Section 2]{moustakas2023descent} for an overview. Also see \cite{billera2006decomposable} for more details about the ribbon Schur functions $s_{Y(\alpha)}$ and \cite{Gelfand1995NCSym} for a discussion of their non-commutative analogue.

By \Cref{lem:desc-rep}, decomposing $\KK[S_k]q_{\hat \alpha}p_\alpha$ into irreducible representations thus amounts to writing a skew Schur function $s_{\lambda/\mu}$ as a sum of standard Schur functions $s_\nu$, which is a classical and well-understood task (see e.g.\ \cite[Exercise 6.19 iv]{Fulton1991RepresentationTA}). In our case, where $\lambda/\mu = Y(\alpha)$ for a composition $\alpha$, there is a particularly simple description via \textit{descent compositions}:

\begin{definition}
    Let $\lambda$ be a partition of $k$. Given a Standard Young Tableau (SYT) $T$ of shape $\lambda$, filled with numbers $1,\ldots,k$, its \textit{descent set} $D(T)$ is the set of $i$ such that $i+1$ is in a lower row than $i$ in $T$. This is a subset of $[k-1]$ and thus naturally corresponds to a composition $\alpha_T$ of $k$, the \textit{descent composition} of $T$.
\end{definition}
For example,
\[\ytableaushort{13,24} \quad \text{and} \quad \ytableaushort{12,34}\]
have descent sets $\{1,3\}$ and $\{2\}$, respectively. They correspond to compositions $(1,2,1)$ and $(2,2)$.

\begin{theorem}[{\cite[Theorem 7]{gessel1984multipartite}}]\label{thm:irrep-dec-char}
    We have
    \[s_{Y(\alpha)}= \sum_{\lambda \vdash k} c_\lambda(\alpha) s_\lambda \]
    where $c_\lambda(\alpha)$ is the number of standard Young tableaux $T$ of shape $\lambda$ with $\alpha_T = \alpha$.
\end{theorem}

\begin{corollary}\label{cor:irred-decomp}
    We have
    \[\KK[S_k]q_{\hat \alpha}p_\alpha \cong \KK[S_k]p_\alpha q_{\hat \alpha} \cong \bigoplus_{\lambda \vdash k} (\mathbb V^\lambda)^{\oplus c_\lambda(\alpha)} \]
    where $\mathbb V^\lambda$ denotes the irreducible $S_k$-representation associated to $\lambda$,
    and consequently
    \[V_\alpha(\KK^d) \cong W_{\hat \alpha}(\KK^d) \cong \bigoplus_{\lambda \vdash k} \mathbb S^\lambda(\KK^d)^{\oplus c_\lambda(\alpha)} \]
    where $\mathbb S^\lambda$ denotes the Schur functor associated to $\lambda$.
\end{corollary}
\begin{proof}
    The first assertion follows from \Cref{thm:irrep-dec-char}. This implies the second assertion by \Cref{prop:descent-alt} and Schur-Weyl duality.
\end{proof}

\begin{example}\label{ex:irrep-decomp}
Let us compute the multiplicities $c_\lambda(\alpha)$ for compositions $\alpha$ of $k=4$. We list all possible SYT by descent compositions in \Cref{tab:tableaux}.

\begin{table}[h]
\centering
\def\arraystretch{2}
    \begin{tabular}{cc|cc}
        $\alpha_T$ & SYT & $\alpha_T$ & SYT \\ \hline
        $(4)$ & $\ytableaushort{1234}$ & $(2,1,1)$ & $\ytableaushort{12,3,4}$ \\
        $(3,1)$ & $\ytableaushort{123,4}$ & $(1,2,1)$ & $\ytableaushort{13,2,4} \quad \ytableaushort{13,24}$ \\
        $(2,2)$ & $\ytableaushort{12,34} \quad \ytableaushort{124,3}$ & $(1,1,2)$ & $\ytableaushort{14,2,3}$\\
        $(1,3)$ & $\ytableaushort{134,2}$ & $(1,1,1,1)$ & $\ytableaushort{1,2,3,4}$\\
    \end{tabular}
    \caption{Young tableaux ordered by descent compositions.}\label{tab:tableaux}
\end{table}

Let us determine the decomposition of $\Pws^4_2(\KK^3)$ and $\Pwa^4_2(\KK^3)$ from this. By \Cref{prop:descent-alt}, we have
\[\Pws^4_2(\KK^3) = \bigoplus_{\alpha\in \comp(4,\leq 2)} V_\alpha(\KK^3) \quad \text{and} \quad \Pwa^4_2(\KK^3) = \bigoplus_{\alpha\in \comp(4,\leq 2)} W_\alpha(\KK^3).\]
Thus, 
\[\Pws_2^4(\KK^3) \cong \mathbb S^{(4)}(\KK^3) \oplus \mathbb S^{(3,1)}(\KK^3)^{\oplus 3} \oplus \mathbb S^{(2,2)}(\KK^3) \]
and 
\[\Pwa_2^4(\KK^3) \cong \mathbb  S^{(2,2)}(\KK^3) \oplus \mathbb S^{(2,1,1)}(\KK^3)^{\oplus 3} \oplus  \mathbb S^{(1,1,1,1)}(\KK^3).\]
Recall from \cite[Theorem 6.3]{Fulton1991RepresentationTA} that the dimension of $\mathbb S^\lambda (\KK^d)$ is
\[\prod_{1\leq i < j \leq d} \left(\frac{\lambda_i-\lambda_j}{j-i} + 1\right)\]
for partitions $\lambda$ of length $\leq d$, and $0$ otherwise. In the example, we compute
\[\dim \Pws^{4}_2(\KK^3) = 15 + 3\cdot 15 + 6 = 66\]
\[\dim \Pwa^{4}_2(\KK^3) = 6 + 3\cdot 3 + 0 = 15\]
which indeed agrees with \Cref{cor:dim-of-pws} (cf.\ \Cref{tab:dim-exmp}).

Let us compare this to the Thrall decomposition of $(\KK^3)^{\tensor 4}$,
\[(\KK^3)^{\tensor 4} = \bigoplus_{\lambda \vdash 4} U_\lambda,\]
where $U_\lambda$ denotes the \textit{Thrall module} associated to the partition $\lambda$. We refer to \cite{AMENDOLA2025} for an overview. \Cref{tab:thrall} shows the decomposition of the $U_\lambda$ into irreducibles. We computed this using the \texttt{Sage} code provided in \cite{AMENDOLA2025}. By comparing \Cref{tab:tableaux} and \Cref{tab:thrall} we see that the Thrall modules $U_\lambda$ do not in general decompose as a direct sum of $V_{\alpha}$, and the spaces $\Pws^4_2(\KK^3)$ and $\Pwa^4_2(\KK^3)$ do not generally decompose into Thrall modules.

\begin{table}[h]
    \centering
    \def\arraystretch{2}
    \begin{tabular}{c|c}
        $\lambda$ & Decomposition of $U_\lambda$\\\hline
        $(4)$ & $\ydiagram{2,1,1} \quad \ydiagram{3,1}$\\
        $(3,1)$ & $\ydiagram{2,1,1} \quad \ydiagram{2,2} \quad \ydiagram{3,1}$\\
        $(2,2)$ & $\ydiagram{1,1,1,1} \quad \ydiagram{2,2}$\\
        $(2,1,1)$ & $\ydiagram{2,1,1} \quad \ydiagram{3,1}$ \\
        $(1,1,1,1)$ & $\ydiagram{4}$ \\
    \end{tabular}
    \caption{Decomposition of summands in the Thrall decomposition of $(\KK^d)^{\tensor 4}$ into irreducible subrepresentations. We represent direct sums of irreducible representations as the collection of corresponding Young diagrams.}
    \label{tab:thrall}
\end{table}
\end{example}

\begin{remark}
    Even though this did not occur in \Cref{ex:irrep-decomp}, the irreducible representations in the decomposition of $V_\alpha$ can also appear with multiplicities. For example, the Young tableaux
    \[\ytableaushort{126,34,5} \quad \text{and} \quad \ytableaushort{124,36,5} \]
    both have descent composition $(2,2,2)$.

    Also note that the representations $V_\alpha$ and $V_\beta$ can be isomorphic for $\alpha \not= \beta$. A precise characterization of this behaviour in terms of $\alpha$ and $\beta$ is given in \cite{billera2006decomposable}. An example is the case where $\beta$ is obtained from $\alpha$ by reversal \cite[Proposition 4.3]{billera2006decomposable}, compare \Cref{tab:tableaux}.

    Finally, the decomposition of descent representations into irreducibles is also studied in \cite[Section 2]{BaRie22} where they are called \textit{Solomon modules}. The authors describe an algorithm to compute this decomposition in general (Algorithm 1 of loc.\ cit.).
\end{remark}

\FloatBarrier

\section{Application to Signatures of Piecewise Linear Paths}\label{sec:signatures}

In this section, we deduce interesting consequences of our main result for the \textit{signature tensors} of piecewise linear paths in $\mathbb R^d$.

A priori, a path in a finite-dimensional vector space is defined through its parametrization $X:[0,T]\to \mathbb R^d$. In this classical picture, a  \textit{piecewise linear path} is such an $X$ for which the derivative $X'$ exists and is constant on each of finitely many time intervals $(0,t_1),(t_1,t_2),\dots, (t_{m-1},t_m)$ which partition $[0,T]$.
However, often we encounter questions where the parametrization does not play any role.
Some of the most obvious attributes of an unparametrized path are its length and the vector from start point to end point.

The general parametrization-independent and translation invariant picture is to view a piecewise linear path merely as a finite sequence of segments\footnote{To be absolutely precise about being unparametrized, we would remove redundancies by requiring that $v_i$ is not a positive scalar multiple of $v_{i-1}$ for any $i$.} $(v_1,...,v_m)\in \mathbb R^{d\times m}$. These derive from the above by $v_i=(t_{i+1}-t_i)X'(t)=X(t_{i+1})-X(t_i)$, $t\in (t_{i-1},t_i)$. Note in particular once we observe the piecewise linear shape, the analytic derivative reduces to an algebraic difference quotient.

While the matrix $(v_1,\dots,v_m)$ obviously has as many columns as the number of segments of the path, the \textit{path signature} provides a way to map this data to a fixed space of sequences of tensors in an essentially injective manner.
Even though the resulting sequences are infinite, finite truncations can be flexibly chosen depending on the purpose to be computationally memory efficient while still getting a powerful description.
Path signatures go back to the works of Chen (\cite{bib:Che1954}, etc), but were later repopularized through Lyons' rough paths (\cite{Lyons98}, etc).
For introductions to the signature method in data science and machine learning, see \cite{chevyrevkormilitzin2016} and \cite{mcleod2025signature}.

To make things precise for the purpose of this article, let $\paths^d_{\leq m}$ denote the set of piecewise linear paths with at most $m$ segments\autonom{\(\paths^d_{\leq m}\)}{The set of piecewise linear paths in $\mathbb R^d$ with at most $m$ segments}, meaning that $X \in \paths^d_{\leq m}$ is given by a sequence $(v_1,\ldots,v_{m'})\in\mathbb R^{d\times m'}$ for $m' \leq m$. The signature $\sigma(X)\in T((\mathbb R^d)):=\prod_{k=0}^\infty (\mathbb R^d)^{\otimes k}$
can then simply be defined in the language of products between tensors, namely by
\begin{equation}\label{eq:def_sig_pwlinear}
    \sigma(X) :=\exp_{\otimes}(v_1)\otimes\cdots\otimes \exp_{\otimes}(v_{m'}),
\end{equation}
where
\begin{equation*}
    \exp_{\otimes }(y):=\sum_{k=0}^\infty \frac{1}{k!} y^{\otimes k}
\end{equation*}
is the formal series exponential such that
\begin{equation*}
    \proj_k \exp_{\otimes }(y)=\frac{1}{k!}\,y^{\otimes k}\in (\mathbb R^d)^{\otimes k}.
\end{equation*}

Note that in the word basis, the inner tensor product $\otimes: T((\mathbb R^d))\times T((\mathbb R^d))\to T((\mathbb R^d))$ is interpreted as the concatenation product
\begin{equation*}
    \Big(\sum_w a_w e_w\Big)\otimes \Big(\sum_{w'} b_{w'} e_{w'}\Big)=\sum_{w}\sum_{w'} a_w b_{w'} e_{ww'}=\sum_{w''}\Big(\sum_{ww'=w''} a_w b_{w'}\Big)e_{w''}.
\end{equation*}
where the sums run over all words in $d$ letters. 
In particular $e_{w_1}\otimes e_{w_2}=e_{w_1w_2}\in (\mathbb R^d)^{\otimes (k_1+k_2)}$ for $e_{w_1}\in (\mathbb R^d)^{\tensor k_1}$, $e_{w_2}\in (\mathbb R^d)^{\tensor k_2}$. 
Note that when we say all words, we mean including the empty word $\emptyword$, indeed $e_{\emptyword}$ spans the zeroth level component $(\mathbb R^d)^{\otimes 0}=\mathbb R$ of the tensor algebra.
Also note that the infinite series appearing are only formal, there is no analytic question of convergence involved. Signatures of piecewise linear paths are thus purely algebraic objects, however an \textit{iterated-integral} perspective will sometimes help the study as we shall see later.

However, sticking to the above definition for now, we may observe the classical properties of path signatures.
It is immediate that we obtain \textit{Chen's identity} \cite[Theorem~3.1]{bib:Che1954}
\begin{equation*}
    \sigma(X\sqcup Y)=\sigma(X)\otimes \sigma(Y)
\end{equation*}
where $X\sqcup Y$ is the concatenation of paths $X$ and $Y$.
If $X$ has segments $(v_1,\dots,v_{m_1})$ and $Y$ has segments $(v'_1,\dots,v'_{m_2})$,
then $X\sqcup Y$ has segments $(v_1,\dots,v_{m_1},v'_1,\dots,v'_{m_2})$.

The main interest in considering the path signature stems from the fact that it essentially characterizes paths up to parametrization and translation.
The slight extension of the equivalence class is so-called tree-like equivalence (see \cite{HL10}, \cite{BGLY16}), which in the multidimensional setting is a usually unproblematic technicality.

In our context of piecewise linear paths, we have the following formulation of that fact.

\begin{proposition}
Let $X,Y$ be piecewise linear paths through $\mathbb R^d$ such that no pair of consecutive segments is colinear. 
Then we have $\sigma(X)=\sigma(Y)$ if and only if the ordered sequence of segments of $X$ equals the ordered sequence of segments of $Y$.
\end{proposition}
Note that the proof should be seen as a technical justification and literature review that can be skipped for a first read.
\begin{proof}
    If $X$ and $Y$ have the same ordered sequence of segments, then obviously $\sigma(X)=\sigma(Y)$ by \eqref{eq:def_sig_pwlinear}.

    So now assume $X$ and $Y$ do not have the same ordered sequence of segments.
    A piecewise linear path such that no pair of consecutive segments is colinear is precisely what is called a nondegenerate piecewise linear path in \cite[Definition~6.1]{HL10} and is in particular a reduced path in the sense of \cite[Corollary~1.6]{HL10}, i.e.\ shortest among all paths with the same signature, by the results of \cite{HL10} taken together.
    Due to their distinct sequence of segments without colinearity of consecutive pairs, $X$ and $Y$ are distinct reduced paths, distinct even after we translate them individually to both start in the origin and arclength parametrize them, and thus $\sigma(X)\neq \sigma(Y)$ by the uniqueness part of \cite[Corollary~1.6]{HL10}.

    In short, the path signature forms an injective map when the domain is reduced to arclength parametrized bounded variation paths that start in the origin and that do not exhibit tree-like excursions. 
\end{proof}

For a first intuition about this statement, note that
\begin{equation*}
    \exp_{\otimes}(v_1)\otimes \exp_{\otimes}(v_2)=\exp_{\otimes}(v_2)\otimes \exp_{\otimes}(v_1)
\end{equation*}
if and only if 
\begin{equation*}
    \exp_{\otimes}(v_1)\otimes \exp_{\otimes}(v_2)=\exp_{\otimes}(v_1+v_2)
\end{equation*}
if and only if $v_1,v_2\in\mathbb R^d$ are colinear. This is a special case of the general formula
\begin{equation*}
    \exp_{\otimes}(v_1)\otimes \exp_{\otimes}(v_2)=\exp_{\otimes}(x)
\end{equation*}
with $x\in T((\mathbb R^d))$ given in terms of iterated Lie-brackets by the Baker-Campbell-Hausdorff formula, see e.g.\ \cite{reutenauer93}. However, $x$ has infinite degree for generic $v_1,v_2$.

Let $T(\mathbb R^d):=\bigoplus_{k=0}^\infty (\mathbb R^d)^{\otimes k}$ denote the space of finite sequences of tensors. The canonical pairing of $T((\mathbb R^d))$ with $T(\mathbb R^d)$ is given by
\begin{equation*}
    \Big\langle\, \sum_w a_w e_w,\,\sum_{w'}b_{w'}e_{w'}\Big\rangle=\sum_{w'}a_{w'}b_{w'},
\end{equation*}
where the sums run over all words, yet only finitely many $b_{w'}$ are non-zero.
There is a commutative associative product $\shuffle$ called the \textit{shuffle} on $T(\mathbb R^d)$, which we will define properly in Section~\ref{sec:ideals}, giving rise to Ree's \cite{Ree58} shuffle identity 
\begin{equation*}
    \langle\sigma(X),x\shuffle y\rangle =\langle \sigma(X),x\rangle \langle \sigma(X),y\rangle.
\end{equation*}
This allows for the famous observation that polynomials in the components of the path signature can be re-expressed as linear forms on the path signature.
The identity is obtained by first checking that it holds for $\exp_\otimes (v), v\in \mathbb R^d$ (or more generally, a \textit{Lie series} as the argument), and subsequently verifying that if it holds for $x,y\in T((\mathbb R^d))$, then it also holds for $x\otimes y$ (yielding the so-called \textit{character group} structure). This is a standard exercise in the general context of connected graded Hopf algebras, see e.g. \cite{reutenauer93} or \cite[Section~3.3]{Preiss16}.

Let us finally observe the following application of this article's previous findings to the level $k$ signature tensors
\begin{equation*}
    \sigma^{(k)}(X):=\proj_k \sigma(X).
\end{equation*}

\begin{proposition}\label{prop:span_of_signatures}
The linear hull of the level $k$ signature tensors for all paths $X$ through $\mathbb R^d$ with (up to) $m$ segments is given by $\Pws^{k}_m(\mathbb R^d)$,
\begin{equation}\label{eq:pwsym}
    \spann \sigma^{(k)}(\paths_{\leq m}^d)=\Pws^{k}_m(\mathbb R^d).
\end{equation}
\end{proposition}
\begin{proof}
Let us write 
\begin{equation*}
    \exp_{\otimes}^{\leq k}(y):=\sum_{i=0}^k\frac{1}{i!}\,y^{\otimes i}
\end{equation*}
and
\begin{equation*}
    \Sym^{\leq k}(\mathbb R^d):=\bigoplus_{i=0}^k \Sym^i(\mathbb R^d)\subseteq \bigoplus_{i=0}^k (\mathbb R^d)^{\otimes i}.
\end{equation*}
 
First of all, we obviously have $\spann_{v}{\exp^{\leq k}_{\otimes}}(v)\subseteq\Sym(\mathbb R^d)^{\leq k}$. To see that even equality holds, note that for any $v\in \mathbb R^d$ we have that
    \begin{equation*}
        \frac{d^i}{dt^i}\exp_{\otimes}(tv)|_{t=0}=v^{\otimes i}.
    \end{equation*}
    Forming that derivative is forming linear combinations inside $\spann_{v}{\exp^{\leq k}_{\otimes}}(v)$ and taking limits, and $\spann_{v}{\exp^{\leq k}_{\otimes}}(v)$ is Euclidean closed as a subspace of the finite dimensional $T(\mathbb R^d)^{\leq k}$. Thus 
    \begin{equation*}
    \Sym(\mathbb R^d)^{\leq k}=\spann\{v^{\otimes i}|v\in\mathbb R^d, i=0,\dots,k\}\subseteq \spann_v\exp_{\otimes}^{\leq k}(v).
    \end{equation*}
    
    To conclude, we finally compute
    \begin{align*}
    \spann \sigma^{(k)}(\paths_{\leq m}^d)&=\spann\{\proj_{k}\exp_{\otimes}(v_1)\otimes\cdots\otimes\exp_{\otimes}(v_m)|v_1,\ldots,v_m\in\mathbb{R}^d\}\\
    &=\spann\{\proj_{k}\exp^{\leq k}_{\otimes}(v_1)\otimes\cdots\otimes\exp^{\leq k}_{\otimes}(v_m)|v_1,\ldots,v_m\in\mathbb{R}^d\}\\
    &=\proj_{k}\spann \{\exp^{\leq k}_{\otimes}(v_1)\otimes\cdots\otimes\exp^{\leq k}_{\otimes}(v_m)|v_1,\ldots,v_m\in\mathbb{R}^d\}\\
    &=\proj_{k}\big(\spann_{v_1}{\exp^{\leq k}_{\otimes}}(v_1)\otimes\cdots\otimes \spann_{v_m}\exp^{\leq k}_{\otimes}(v_m)\big)\\
    &=\proj_{k}\big(\Sym(\mathbb R^d)^{\leq k}\otimes \cdots \otimes \Sym(\mathbb R^d)^{\leq k}\big)\\
    &=\sum_{\substack{k_1+\ldots +k_\ell=k,\\\ell\leq m}}\Sym^{k_1}(\mathbb R^d)\otimes\cdots\otimes \Sym^{k_\ell}(\mathbb R^d)\\
    &= \Pws^{k}_m(\mathbb R^d)
    \end{align*}
    Regarding the first equality, note that for less than $m$ segments, we can just choose some $v_i$ to be zero. 
\end{proof}

We are equally interested in the \textit{relations} for signatures of piecewise linear paths (cf.\ \cite[Section~1]{Pre24})
\begin{align*}
    \mathcal{I}(\paths^d_{\leq m})&:=\{y\in T(\mathbb R^d)\,|\,\langle \sigma(X),y\rangle =0\, \forall X\in\paths^d_{\leq m}\}\\
    &\,\,=\{y\in T(\mathbb R^d)\,|\,\langle x,y\rangle=0\,\forall x\in {\textstyle \bigoplus_{k=0}^{\infty}\spann\sigma^{(k)}(\paths^d_{\leq m})}\}.
\end{align*}\autonom{\(\mathcal{I}(\paths^d_{\leq m})\)}{The vanishing ideal of $\sigma(\paths^d_{\leq m})$ in the tensor algebra $T(\mathbb R^d)$}
Note that $T((\mathbb R^d))$ is the algebraic dual space of $T(\mathbb R^d)$ (however, $T(\mathbb R^d)$ is of course not the algebraic dual space of $T((\mathbb R^d))$!),
so $\mathcal{I}(\paths^d_{\leq m})$ is the predual annihilator of $\sigma(\paths^d_{\leq m})$. 
The subspace $\mathcal{I}(\paths^d_{\leq m})$ was first explicitly introduced in \cite{lotter2025cyclic} at the very end of Section~2.

From Theorem \ref{thm:main} and Proposition \ref{prop:span_of_signatures}, we immediately obtain the following.
\begin{corollary}\label{cor:AnnSig}
The set of all linear relations on level $k$ signature tensors of paths of $m$ segments is the set of piecewise alternating tensors with $k-m$ connected blocks,
\begin{equation*}
\proj_k\,\mathcal{I}(\paths^d_{\leq m})= \{ y\in(\mathbb R^d)^{\otimes k}\,|\,\langle x,y\rangle=0\,\,\,\forall x \in \sigma^{(k)}(\paths^d_{\leq m})\}=\Pwa^{k}_{k-m}(\mathbb R^d).
\end{equation*}
As a graded linear subspace of $T(\mathbb R^d)$,
\begin{equation*}
\mathcal{I}(\paths^d_{\leq m})=\bigoplus_{k=0}^\infty \Pwa^{k}_{k-m}(\mathbb R^d),
\end{equation*}
where we from now on additionally put $\Pwa^{k}_{k-m}(\mathbb R^d):=\{0\}$ for $k-m<0$.
\end{corollary}

\begin{remark}
    We skip any details here, but want to note the following for the interested reader.
    Often one considers instead of the signature of a path $X$ its \textit{log-signature}
    \begin{equation*}
        \log_{\otimes}(\sigma(X)),
    \end{equation*}
    which lives in the completed free Lie algebra. 
    Then, analogously, we can ask what is
    \begin{equation*}
        \spann\proj_k\log_{\otimes}\sigma(\paths^d_{\leq m})
    \end{equation*}
    as well as what is its annihilator.
    The latter can be described by first intersecting $\mathcal{I}(\paths^d_{\leq m})$ with the graded linear subspace of $T(\mathbb R^d)$ given by \textit{coordinates of the first kind}, and then apply the canonical bijection between coordinates of the first kind and the dual of the free Lie algebra.
    Finally, we obtain the linear span by taking the annihilator of these relations.
\end{remark}
Let us from now on write
\begin{equation*}
    \vol(w):=\sum_{\tau \in S_{|w|}} \mathrm{sgn}(\tau) e_{\tau(w)}
\end{equation*}
for the antisymmetrization we already encountered before.

This is motivated by the fact that for a sequence $w=i_1\cdots i_k$ with $i_j\in\{1,\dots,d\}$ and a path $X\in\paths^d_{\leq m}$,
\begin{equation*}
    \langle \sigma(X),\vol(w)\rangle
\end{equation*}
is $k!$ times the \textit{signed volume} enclosed by the path $(X_{i_1},\ldots,X_{i_k})\in\paths^{k}_{\leq m}$, see \cite{DR18}. In particular, if $(X_{i_1},\ldots,X_{i_k})$ is a convex path, then it is $k!$ times the volume of the convex hull of the image of $(X_{i_1},\ldots,X_{i_k})$, cf.\ \cite{amendolaleemeroni23}.

While the length of a piecewise linear path does not depend on the ``direction of travel'',
the full increment vector already mentioned above does.
The signed volumes even depend on the ``sequence of events'', i.e.\ they are not in general invariant under permutations of the segments even if they preserve the image of the path\footnote{Non-trivial permutations of segments that preserve the image may occur when the path has self-intersections.}.

The main theorem of this article can be re-expressed using concatenations of signed volumes.

\begin{theorem}\label{thm:conc_vol}
    The set of all 
    \begin{equation*}
    \vol(w_1)\otimes \cdots \otimes \vol(w_\ell)
    \end{equation*}
    where $k_1+\ldots+k_\ell=k$, $d\geq k_i\geq 1$, $k-\ell+1>m$ and $w_i$ strictly decreasing with $|w_i|=k_i$ forms a linear generating set of $ \Pwa^{k}_{k-m}(\mathbb R^d)$.
    \end{theorem}
    Examples for $d=3,m=5,k=6$ and $d=4,m=6,k=7$ were already given in \cite[Example~4.6 and Section~5]{lotter2025cyclic}.

    Let us provide an alternative geometric argument of why the concatenations of volumes from Theorem \ref{thm:conc_vol} are elements of $\proj_k \mathcal I(\paths^d_{\leq m}) = \Pwa^{k}_{k-m}(\mathbb R^d)$.
    This also goes back to \cite[Example~4.6]{lotter2025cyclic}, but is given here for the first time in complete generality.
    
    Given a piecewise linear path $X$, the signature can also be presented in integral form, i.e.\
    \begin{equation*}
        \langle\sigma(X),e_{i_1\cdots i_k}\rangle=\int_{0\leq t_1\leq \cdots\leq t_k\leq T} X_{i_1}'(t_1)\cdots X'_{i_k}(t_k)\,\mathrm dt_1\cdots \mathrm dt_k.
    \end{equation*}
    The integral form is actually much more general and works for bounded variation paths and even beyond, if we replace the derivatives by Stieltjes integrals, see e.g.\ \cite{FV10}.

    \begin{definition}\label{def:sigintegrand}
        For $z=\sum_{i_1,\dots,i_k=1}^d a_{i_1\cdots i_k}e_{i_1\cdots i_k}\in (\mathbb R^d)^{\otimes k}$ and a piecewise linear path $X$, let $\sigintegrand{X}{z}:\,\Delta_T^k\to \mathbb R$ denote the integrand of $\langle \sigma(X),z\rangle $ given by
        \begin{equation*}
        \sigintegrand{X}{z}(t_1,\dots,t_k):=\sum_{i_1,\dots,i_k=1}^d a_{i_1\cdots i_k} X_{i_1}'(t_1)\cdots X_{i_k}'
(t_k)        \end{equation*}
for all $(t_1,\dots,t_k)\in \Delta_T^k$ such that all the derivatives exist, and let  $\sigintegrand{X}{z}(t_1,\dots,t_k)$ be zero otherwise.
    \end{definition}
    For $w=i_1\cdots i_k$ and $v_1,\ldots, v_k \in \mathbb R^d$, $v_j=(v_{j,1},\dots,v_{j,d})^\top$, put
    \begin{equation*}
        {\det}_{w}(v_1,\dots,v_k):=\det\begin{pmatrix}v_{1,i_1}& \cdots &v_{k,i_1}\\\vdots& & \vdots\\
        v_{1,i_k}&\cdots & v_{k,i_k}\end{pmatrix}.
    \end{equation*}
    Now let $\vol(w_1)\otimes \cdots \otimes \vol(w_\ell)$ be a concatenation of volumes, let $k_i$ be the length of $w_i$ and set  $k^+_j:=k_1+\ldots+k_j$.
    Then we obtain
    \begin{small}
    \begin{align*}
        \sigintegrand{X}{\vol(w_1)\otimes \cdots \otimes \vol(w_\ell)}={\det}_{w_1}(X'({t_1}),\dots, X'(t_{k_1}))\cdots{\det}_{w_\ell} (X'({t_{k-k_{\ell}+1}}),\dots, X'(t_{k})).
    \end{align*}
    \end{small}If this expression does not vanish, then $(X'(t_{k^+_j+1}),\dots,X'(t_{k^+_j+k_{j+1}}))$ must be linearly independent for all $0\leq j \leq \ell-1$. 
    In particular, there must be $t_{k^+_j+1}\leq\dots\leq t_{k^+_j+k_{j+1}}$ such that $X'(t_{k^+_j+1}),\dots,X'(t_{k^+_j+k_{j+1}})$ are pairwise distinct segments.
    But since also $t_{k^+_j+k_{j+1}}=t_{k^+_{j+1}}\leq t_{k^+_{j+1}+1}$ for each $j$, this means that
    $X$ must have at least
    \begin{equation*}
    k_\ell + \sum_{i=1}^{\ell -1} (k_i - 1) =k-\ell+1
    \end{equation*}
    segments.
    Therefore, conversely, if $X$ has $m$ segments with $m<k-\ell+1$, then 
    \begin{equation*}
        \langle\sigma(X),\vol(w_1)\otimes \cdots \otimes \vol(w_\ell)\rangle=0.
    \end{equation*}
    Thus 
    \begin{equation*}
        \vol(w_1)\otimes \cdots \otimes \vol(w_\ell)\in \proj_{k_\ell^+}\,\mathcal{I}(\paths^d_{\leq m}).
    \end{equation*}
Let us illustrate this in an example.
\begin{example}
    Consider the concatenation of volumes $\vol(1,2) \tensor \vol(1,2)$. If $X$ is a path with two segments, then $\det(X'(t_1) \ X'(t_2))\not=0$ forces $X'(t_3) = X'(t_4)$ as visualized in \Cref{fig:vol-van}. Thus, we always have $\det(X'(t_1) \ X'(t_2)) \det(X'(t_3) \ X'(t_4)) = 0$, and thus $\langle \sigma(X), \vol(1,2) \tensor \vol(1,2)\rangle = 0$.
    \end{example}

    \begin{figure}[h]
    \centering
    \includegraphics[width=0.5\linewidth]{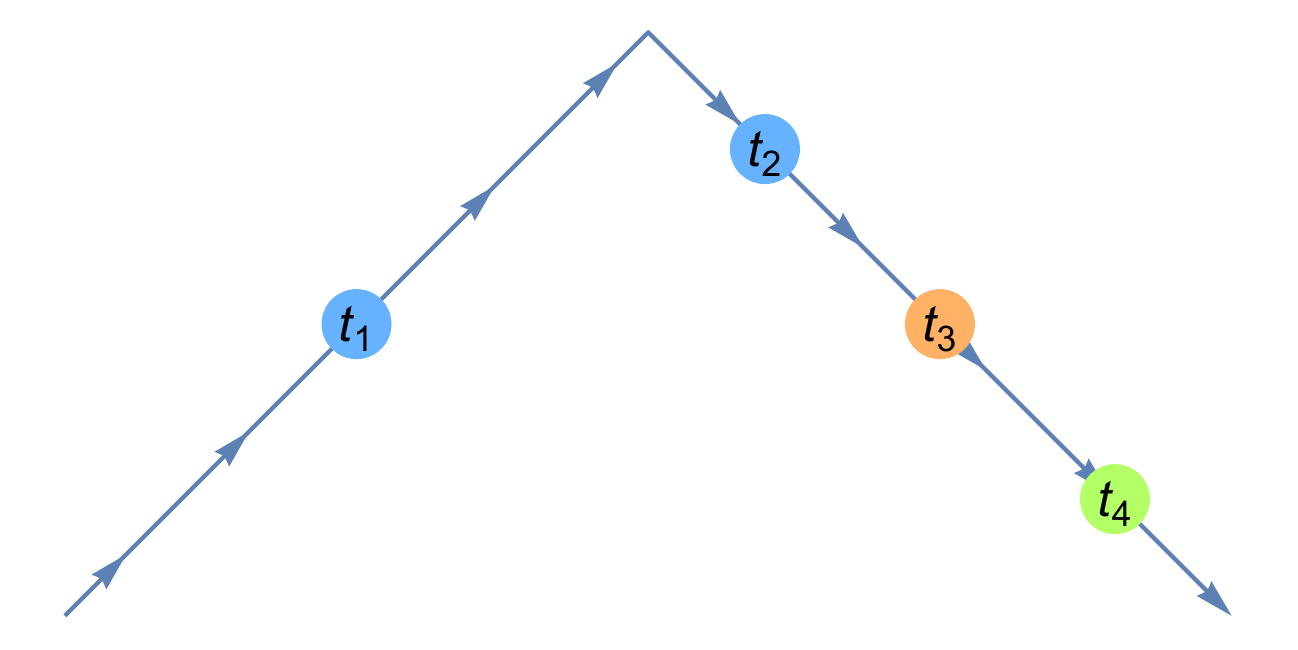}
    \caption{It is not possible to find $t_1\leq \ldots \leq t_4$ with $X'(t_1) \not= X'(t_2)$ and $X'(t_3) \not= X'(t_4)$ on a two-segment path.}\label{fig:vol-van}
\end{figure}

This argument can be extended from concatenations of volumes to \textit{interlacings of volumes}.
While any concatenation of signed volumes of the form
\begin{equation*}
    \vol(i_1i_2i_3)\otimes\vol(i_4i_5i_6)
\end{equation*}
vanishes under the signature when we fail to find segments $v'_1,v_2',v_3',v'_4,v_5'$ in the correct order from $(v_1,\dots,v_m)$ such that $v'_1, v_2',v_3'$ and $v'_3,v'_4,v_5'$ are linearly independent,
any interlacing of the volumes $\vol(i_1i_2i_3i_4)$ and $\vol(i_1i_2i_3)$ of the form

\begin{equation*}
    \sum_{\tau\in S_4}\mathrm{sgn}(\tau)\,e_{i_{\tau(1)}i_{\tau(2)}}\otimes\vol(i_5i_6i_7)\otimes e_{i_{\tau(3)}i_{\tau(4)}}
\end{equation*}
vanishes under the signature when we fail to find segments $v'_1,v_2',v_3',v'_4,v_5'$ in the correct order from $(v_1,\dots,v_m)$ such that $v'_1,v_2',v_4',v_5'$ and $v'_2,v'_3,v_4'$ are linearly independent.
Both expressions vanish when we have less than five segments altogether.
It might seem surprising at first that interlacings of volumes do not yield more relations on $\spann\sigma(\paths^d_{\leq m})$ than linear combinations of concatenations of volumes.
However, this is another simple consequence of the main result of this article.

The reader is advised to not lose themselves in the technical details of the following definition, we will not need it for the last section.
It is just meant to introduce interlacings of volumes generally and rigorously to make Corollary \ref{cor:interlacings} a precise and satisfying statement.
\begin{definition}
For $k, m, d\in \mathbb N$, let $\mathcal S_{d,m,k}$ denote the set of all set partitions $\mathscr I$ of $[k]$ with block size at most $d$ such that there is no non-decreasing map $f: [k] \to [m]$ with $f|_I$ injective for all $I \in \mathscr I$.

        For a partition $\mathscr I = \{I_1, \ldots, I_l\} \in \mathcal S_{d,m,k}$, write $S_{\mathscr I} := S_{I_1} \times \ldots \times S_{I_l}$.
We call a word $w = i_1 \cdots i_k$ of length $k$ in letters $1,\ldots,d$ an \textit{$\mathscr I$-essential} word if for all $I \in \mathscr I$, $I=\{j_1 < \cdots < j_r\}$, the induced word $w|_{I}:=i_{j_1}\cdots{i}_{j_r}$ contains no letter more than once.
$S_I$ acts by permutation of letters at indices $i\in I$ on words of length $k$.
\end{definition}
Let us then denote the partial antisymmetrization
$$\sum_{\tau \in S_{\!\mathscr I}} \mathrm{sgn}(\tau) \, e_{\tau(w)}$$ by $\mathscr I(w)$.
Then, once again, we must have 
\begin{equation*}
    \langle\sigma(X),\mathscr{I}(w)\rangle=0
\end{equation*}
    for $\mathscr{I}\in\mathcal S_{d,m,k}$, an $\mathscr{I}$-essential word $w$ of length $k$ and $X$ a piecewise linear path through $\mathbb R^d$ with $m$ segments.
    Thus, by Corollary \ref{cor:AnnSig} and Theorem \ref{thm:conc_vol}, we conclude the following.
\begin{corollary}\label{cor:interlacings}
        Let $\mathscr{I}\in\mathcal{S}_{d,m,k}$ and let $w$ be an $\mathscr{I}$-essential word of length $k$. Then
        \begin{equation*}
            \mathscr{I}(w)\in \proj_k\,\mathcal{I}(\paths^d_{\leq m}) = \Pwa_{k-m}^{k}(\mathbb R^d).
        \end{equation*}
        In other words, any such interlacing of volumes $\mathscr{I}(w)$ can be written as a linear combination of concatenations of volumes of the form
        \begin{equation*}
            \vol(w_1)\otimes\cdots\otimes \vol(w_\ell)
        \end{equation*}
        with $\ell\leq k-m$.
    \end{corollary}

\section{Ideal structures on piecewise alternating tensors}\label{sec:ideals}

Given the dimension of the ambient space of the path $d$ and the number of path segments $m$, so far we have treated the $\Pwa_{k-m}^{k}(\mathbb R^d)=\ann\sigma^{(k)}(\paths^d_{\leq m})$ independent of each other for each $k$ individually.
However, it is for example immediately clear from Theorem \ref{thm:conc_vol} that if we found a relation $x\in\Pwa_{k-m}^{k}(\mathbb R^d)$, then $x\otimes e_{i}$ and $e_{i}\otimes x$ are relations in $\Pwa_{k+1-m}^{k+1}(\mathbb R^d)$ for any letter $i=1,\dots,d$.
To investigate this further, let us look more specifically again at the full graded vector space of relations
\begin{equation*}
    \mathcal{I}(\paths^d_{\leq m})=\bigoplus_{k=0}^\infty\Pwa_{k-m}^{k}(\mathbb R^d).
\end{equation*}
Then the stability with respect to concatenation from the left or from the right that we just observed means precisely that $\mathcal{I}(\paths^d_{\leq m})$ is a \textit{two-sided ideal} (see e.g. \cite{brevsar2014introduction} or \cite{li2002noncommutative}) with respect to the concatenation product $\otimes$ that makes $T(\mathbb R^d)$ the free associative algebra.

However, to understand the structure of the ideal even better, and in particular also in a more geometrical manner, we need to introduce the shuffle and halfshuffle products, which go back to \cite{EM53}, \cite{Ree58} and \cite{S58}. For this purpose, let $T^{\geq 1}(\mathbb{R}^d):=\bigoplus_{k=1}^\infty (\mathbb R^d)^{\otimes k}$ denote the non-unital part of the tensor algebra.

\begin{definition}
The bilinear right $\succ$ and left $\prec$ \textit{halfshuffles} $\succ,\prec:T^{\geq 1}(\mathbb{R}^d)\times T^{\geq 1}(\mathbb{R}^d)\to T^{\geq 1}(\mathbb{R}^d)$ are recursively given by
 \begin{align*}
  e_{w}\succ e_i &:= e_{wi},&e_i\prec e_w &:= e_{iw}\\
  e_{w_1}\succ e_{w_2i} &:= (e_{w_1}\succ e_{w_2}+e_{w_2}\succ e_{w_1})\otimes e_i,\quad& e_{iw_2} \prec e_{w_1}&:= e_i\otimes(e_{w_1}\prec e_{w_2}+e_{w_2}\prec e_{w_1})
 \end{align*}
 for arbitrary non-empty words $w,w_1,w_2$ and arbitrary letters $i$.
    Finally, the \textit{shuffle product} is defined as the symmetrization of either halfshuffle,
    \begin{equation*}
        x\shuffle y:=x\succ y+y\succ x=x\prec y+y\prec x
    \end{equation*}
    and extended to a bilinear product $\shuffle:\,T(\mathbb R^d)\times T(\mathbb R^d)\to T(\mathbb R^d)$ via
    \begin{equation*}
        \emptyword \shuffle x:=x\shuffle \emptyword:= x.
    \end{equation*}
\end{definition}\autonom{\(\succ,\, \prec,\, \shuffle\)}{The right and left halfshuffle and the shuffle product}

Note that while the shuffle product is associative and commutative, 
the halfshuffles are non-commutative and non-associative. 
It is a classical result (see e.g. \cite{reutenauer93}, generally a good introduction to shuffle products) 
that $(T(\mathbb R^d),\shuffle)$ is a free commutative algebra, 
and thus can be thought of as a polynomial ring in infinitely many variables. 
$(T^{\geq 1}(\mathbb R^d),\succ)$ and $(T^{\geq 1}(\mathbb R^d),\prec)$ by contrast are the free left and right \textit{Zinbiel algebra}, respectively (see \cite[page~19]{S58} and \cite[Proposition~1.8]{L95}), in finitely many generators $e_1,\dots,e_d$.
As in \cite{Pre24}, we warn the reader not to confuse the combination of $\succ$ and $\prec$ with a noncommutative dendriform algebra.

With this terminology, we observe the following. 
\begin{proposition}
    $\mathcal{I}(\paths^d_{\leq m})$ is an ideal with respect to the right and left halfshuffle product, and in particular with respect to the shuffle product.
    That means $\mathcal{I}(\paths^d_{\leq m})\subseteq T^{\geq 1}(\mathbb R^d)$ and
\begin{align*}
        x\in T^{\geq 1}(\mathbb R^d),\, y\in \mathcal{I}(\paths^d_{\leq m}) &\Rightarrow x\succ y,\, y\succ x\in\mathcal{I}(\paths^d_{\leq m}),\\
        x\in T^{\geq 1}(\mathbb R^d),\, y\in \mathcal{I}(\paths^d_{\leq m}) &\Rightarrow x\prec y,\, y\prec x\in\mathcal{I}(\paths^d_{\leq m}),\\
        x\in T(\mathbb R^d),\, y\in \mathcal{I}(\paths^d_{\leq m}) &\Rightarrow x\shuffle y\in\mathcal{I}(\paths^d_{\leq m}).
    \end{align*}
\end{proposition}
\begin{proof}
    $\mathcal{I}(\paths^d_{\leq m})$ is a graded subspace of $T(\mathbb R^d)$ by Corollary \ref{cor:AnnSig} and for the empty word, we have $e_\emptyword\notin \mathcal{I}(\paths^d_{\leq m})$ since $\langle\sigma(X),e_\emptyword\rangle=1$ for all paths $X$, thus $\mathcal{I}(\paths^d_{\leq m})\subseteq T^{\geq 1}(\mathbb R^d)$.
    As the family of piecewise linear paths up to a certain number of segments has the crucial property that it is stable under arbitrary restriction to subintervals of the time domain, the statement of the proposition is a direct application of \cite[Corollary~3.5]{Pre24}.
\end{proof}

In the case of two-segment paths, by our computations, the structure seems to be particularly simple.
\begin{conjecture}
    $\mathcal{I}(\paths^d_{\leq 2})$ as an ideal with respect to left and right halfshuffle products is generated in levels $k=3,4$.
\end{conjecture}

Note that we know that only $k=3$ or only $k=4$ is not enough here for $d\geq 3$, as we have $\vol(123)\in\Pwa^3_{3-2}$, yet e.g.\ $\vol(12)\tensor\vol(12)\in \Pwa^4_{4-2}$ cannot be obtained from $\vol(123)$ by combinations of halfshuffle operations.

Unfortunately, we were not yet able to even make a good guess about whether all $\mathcal{I}(\paths^d_{\leq m})$ are finitely generated as ideals with respect to the left and right halfshuffle products. However, as we will see now, $\mathcal{I}(\paths^d_{\leq m})$ admits an even stronger ideal structure for which we can obtain a finite generation result.

\begin{definition}
    We call a graded linear subspace $I\subseteq T(\mathbb R^d)$ a \textit{letter-insertion ideal} if for any $k\in\mathbb N$, any homogeneous $x \in (\mathbb R^d)^k$, all $s\in \{0,\dots,k\}$ and all $j\in\{1,\dots,d\}$ we have
    \begin{align*}
      x\in I \Rightarrow \insertletter(x,s,j)\in I,
    \end{align*}
    where
    \begin{equation*}
        \insertletter\Big(\sum_{i_1,\dots,i_k=1}^d a_{i_1\cdots i_k}e_{i_1\cdots i_k}, s,\,j\Big):=\sum_{i_1,\dots,i_k=1}^d a_{i_1\cdots i_k}e_{i_1\cdots i_{s}ji_{s+1}\cdots i_k}
    \end{equation*}
    defines the insertion of the letter $j$ directly after the $s$-th letter (cf.\ \cite[Section~3.3]{DR18}).
\end{definition}

This means that given a linear combination of words of the same length, we are allowed to insert a letter \text{at the exact same position} into each of the words.

Any letter-insertion ideal is of course a two-sided ideal with respect to $\otimes$,
since we are allowed to choose all $x_j$ or all $y_j$ to be the empty word. Furthermore, any letter-insertion ideal satisfies the \textit{insertion of factors property (IFP)} (terminology from e.g.\ \cite{hashemi2008} and going back to \cite{bell1970}) given by
\begin{equation*}
      x\otimes  y\in I, \,z\in T(\mathbb R^d)\Rightarrow x\otimes z\otimes y\in I.
\end{equation*}

Finally, any letter-insertion ideal is an ideal with respect to left and right halfshuffle products, and in particular with respect to the shuffle product.
This follows immediately from the recursive definitions.

Now, recall the notation $\sigintegrand{X}{z}$ introduced in \Cref{def:sigintegrand}.
An important way in which \mbox{letter-insertion} ideals arise is then the following.

\begin{proposition}\label{prop:vanishing-integrands-shconc-ideal}
    Let $P\subseteq \paths^d$ and $I\subseteq T(\mathbb R^d)$ be the graded linear subspace spanned by all homogeneous $z\in (\mathbb R^d)^{\otimes k}, k\in\mathbb N_0$ such that we have vanishing of integrands
    \begin{equation*}
        \sigintegrand{X}{z}(t_1,\dots,t_k)=0
    \end{equation*}
    for almost all $(t_1,\dots,t_k)\in\Delta_T^k$ and all $X\in P$. 
    Then $I$ is a letter-insertion ideal.
\end{proposition}
\begin{proof}
   Let $k\in \mathbb N_0$ and $z=\sum_{i_1,\dots,i_k=1}^d a_{i_1\cdots i_k}e_{i_1\cdots i_k}\in(\mathbb R^d)^{\otimes k}\cap I$ arbitrary. Since 
    \begin{equation*}
        0=\sigintegrand{X}{z}(t_1,\dots,t_k)=\sum_{i_1,\dots,i_k=1}^da_{i_1\cdots i_k} X_{i_1}'(t_1)\cdots X_{i_k}'(t_k)
    \end{equation*}
    for almost all $(t_1,\dots,t_k)\in\Delta_T^k$,
    we then obviously have

    \begin{align*}
        &\sigintegrand{X}{\insertletter(z,s,j)}(t_1,\dots,t_k)\\
        &=\sum_{i_1,\dots,i_k=1}^da_{i_1\cdots i_k} X_{i_1}'(t_1)\cdots X_{i_s}'(t_s)X_{j}'(t_{s+1})X_{i_{s+1}}'(t_{s+2})\cdots X_{i_k}'(t_{k+1})\\
        &=X_{j}'(t_{s+1})\sum_{i_1,\dots,i_k=1}^da_{i_1\cdots i_k} X_{i_1}'(t_1)\cdots X_{i_s}'(t_s)X_{i_{s+1}}'(t_{s+2})\cdots X_{i_k}'(t_{k+1})=0
    \end{align*}
    for all $s\in\{0,\dots,k\}$, $j\in\{1,\dots,d\}$ and almost all $(t_1,\dots,t_{k+1})\in\Delta_T^{k+1}$.
    Thus, any result of a letter insertion into $z$ is again in $I$.
\end{proof}

Finally, with the following result, a relatively simple description of the structure of $\mathcal{I}(\paths^{d}_{\leq m})$ is possible.
\begin{proposition}
    $\mathcal{I}(\paths^d_{\leq m})$ is a finitely generated letter-insertion ideal.
    More precisely, for $d\geq 2$, it is generated in levels $k_0^{d,m},\ldots,2m$, where $k_0^{d,m}=\lceil\tfrac{dm}{d-1}\rceil$.
 \end{proposition}
\begin{proof}
    For $d=1$, we always have $\mathcal{I}(\paths^1_{\leq m})=\{0\}$, in particular it is a letter-insertion ideal and finitely generated.
    So assume $d\geq 2$ for the rest of the proof.
    That $\mathcal{I}(\paths^d_{\leq m})$ is a letter-insertion ideal then follows immediately from Proposition \ref{prop:vanishing-integrands-shconc-ideal} and the discussion about concatenation of volumes in Section \ref{sec:signatures}.
    The lowest degree $k_0^{d,m}$ is by Corollary \ref{cor:space-spanning-bound} the smallest integer $k$ such that $\lfloor\tfrac{d-1}{d}k\rfloor\geq m$, which is $\lceil\tfrac{dm}{d-1}\rceil$.
    
    We are left to show that we do not need generators from levels higher than $2m$. 
    Thus, let $k>2m$, and
    \begin{equation*}
     \vol(w_1)\otimes \cdots\otimes \vol(w_\ell)\in\Pwa_{k-m}^{k}(\mathbb R^d)
     \end{equation*}
     with $\ell\leq k-m$ arbitrary.
     
     If $\ell=k-m$, there is at least one $w_j$ which is a letter, since $2\ell=2k-2m>2k-k=k$. Thus $\vol(w_1)\otimes \cdots\otimes \vol(w_\ell)$ is induced by 
     \begin{equation*}
    \vol(w_1)\otimes \cdots\otimes\vol(w_{j-1})\otimes \vol(w_{j+1})\otimes\cdots\otimes \vol(w_\ell)\in\Pwa_{k-1-m}^{k-1}(\mathbb R^d).
     \end{equation*}

     If otherwise $\ell<k-m$, then either $w_\ell$ is a letter, allowing us to argue like in the previous step, or 
     \begin{equation*}
     \vol(w_\ell)=\sum_i s_i\vol(w_i')\otimes e_{i},
     \end{equation*}
     where the sum runs over the letters appearing in $w_\ell$, and  $w_i'$ and $s_i$ are such that there is $\tau\in S_{|w_\ell|}$ with $\operatorname{sgn} \tau=s_i$ and $\tau(w_\ell)=w_i'i$ (cf.\ \cite[Lemma~3.17]{DR18}). Thus, $\vol(w_1)\otimes \cdots\otimes \vol(w_\ell)$ is induced by the corresponding
     \begin{equation*}
         \vol(w_1)\otimes\cdots\otimes\vol(w_{\ell-1})\otimes\vol(w_i')\in \Pwa_{k-1-m}^{k-1}(\mathbb R^d).
     \end{equation*}
\end{proof}

\printnomenclature[4cm]

\bibliographystyle{alpha}
\bibliography{main}
\end{document}